\documentclass[11pt]{article}
\usepackage{amsmath,amsfonts,amssymb,amsthm,verbatim}
\usepackage{graphicx}
\usepackage[mathcal]{euscript}
\usepackage{color}

\usepackage{tikz}

\usetikzlibrary{positioning,calc,arrows}

\usepackage{palatino}

\def\R{\mathbb{R}}
\def\C{\mathbb{C}}

\def\RP{\mathbb{R}P}

\def\ve{\mathsf{e}}
\def\ebar{\overline{\ve}}
\def\vv{\mathsf{v}}
\def\vw{\mathsf{w}}

\def\vK{\mathsf{K}}
\def\vV{\mathsf{V}}
\def\vW{\mathsf{W}}
\def\vX{\mathsf{X}}
\def\vY{\mathsf{Y}}
\def\vZ{\mathsf{Z}}

\def\vn{\mathsf{n}}
\def\mF{\EuScript{F}}   
\def\FR{\mF_{\R^3}}
\def\FRm{\mF_{\R^{1,2}}}
\def\FM{\mF_M}
\def\M{\mathcal{M}}
\def\ri{\mathrm{i}}
\def\I{\mathcal{I}}
\def\calJ{\mathcal{J}}
\def\calK{\mathcal{K}}

\def\di{\partial}
\def\w{\omega}
\def\Lie{\mathcal{L}}
\def\chip{\mbox{\raisebox{.35ex}{$\chi$}}}
\def\eps{\varepsilon}
\def\ann{\operatorname{ann}}
\def\inv{\operatorname{inv}}
\def\rk{\operatorname{rank}}
\def\csch{\operatorname{csch}}
\def\sech{\operatorname{sech}}

\def\V{\mathcal V}
\def\pV{V}      

\def\proj{{\pi}}

\def\C #1{\mathcal{#1}}
\def\B #1{\mathbb{#1}}
\def\P #1{\partial_{#1}}
\def\ds{\displaystyle}
\def\bsy{\boldsymbol}

\def\we{\wedge}

\newtheorem{thm}{Theorem}
\newtheorem{prop}[thm]{Proposition}
\newtheorem{lemma}[thm]{Lemma}
\theoremstyle{definition}
\newtheorem{rem}{Remark}
\newtheorem*{rem*}{Remark}
\newtheorem*{rems*}{Remarks}
\newtheorem{xmpl}{Example}
\newtheorem{defn}{Definition}

\def\wh{\widehat}
\def\bsy{\boldsymbol}

\def\we{\wedge}

\def\CB{{\mathcal{B}}}

\def\CV{{\mathcal{V}}}
\def\CH{{\mathcal{H}}}

\def\CN{\mathcal{N}}
\def\CW{{\mathcal{W}}}

\def\CI{{\mathcal{I}}}

\def\B#1{\mathbb{#1}}

\def\a{\alpha}
\def\b{\beta}

\def\e{\varepsilon}
\def\g{\gamma}

\def\O{\Omega}

\def\o{\omega}
\def\O{\Omega}

\def\s{\sigma}

\def\th{\theta}

\def\mg{\bsy{g}}  
\def\brho{\bsy{\rho}}

\def\P #1{\partial_{#1}}
\def\ds{\displaystyle}

\usepackage{bm}

\newcommand{\redmark}[1]{\textcolor{red}{#1}}

\begin{document}
\title{Isometric Embedding and Darboux Integrability}
\author{J.N. Clelland,\footnote{Department of Mathematics, University of Colorado, Boulder, CO., USA; {\tt jeanne.clelland@colorado.edu}}
\ \ T.A. Ivey,\footnote{Department of Mathematics, College of Charleston, Charleston, SC., USA; {\tt iveyt@cofc.edu}}
 \ \ N. Tehseen,\footnote{Department of Mathematics and Statistics,
La Trobe University, P.O. Box 199, Bendigo, VIC 3552, Australia; {\tt  naghmanatehseen@gmail.com}}
\ \ P.J. Vassiliou.\footnote{Department of Theoretical Physics, Australian National University, Canberra,  A.C.T., Australia, 2601; {\tt peter.vassiliou@canberra.edu.au}}}
\date{\today}
\maketitle
\begin{abstract}{\small  Given a smooth 2-dimensional Riemannian or pseudo-Riemannian manifold $(M, \bsy{g})$ and an ambient 3-dimensional Riemannian or pseudo-Riemannian manifold $(N, \bsy{h})$, one can ask under what circumstances does the exterior differential system  $\I$ for the isometric embedding $M\hookrightarrow N$ have particularly nice solvability properties.
In this paper we give a classification of all $2$-metrics $\bsy{g}$ whose local isometric embedding system into flat Riemannian or pseudo-Riemannian 3-manifolds $(N, \bsy{h})$ is Darboux integrable. As an illustration of the motivation behind the classification, we examine in detail one  of the classified metrics, $\bsy{g}_0$,  showing how to use its Darboux integrability in order to construct all its embeddings in finite terms of arbitrary functions.  Additionally, the geometric Cauchy problem for the embedding of $\bsy{g}_0$ is shown to be reducible to a system of two first-order ODEs for two unknown functions---or equivalently, to a single second-order scalar ODE.
For a large class of initial data, this reduction permits explicit solvability of the geometric Cauchy problem for $\bsy{g}_0$ up to quadrature. The results described for  $\bsy{g}_0$ also hold for any classified metric whose embedding system is hyperbolic.
\vskip 0.25mm

}\end{abstract}

\baselineskip=12 pt

\section{Introduction}

It is interesting to wonder under what circumstances a given non-trivial geometric problem has particularly ``nice" or special solutions. For instance there is now a beautiful theory of {\it integrable} geodesic flows that has developed particularly over the last 30-40 years; {\it see} \cite{Kiyohara}. In this paper we address this type of question in relation to isometric embedding of one Riemannian or pseudo-Riemannian manifold into another. Specifically, we study the {\it integrability} of the exterior differential system $\C I$ for the isometric embedding problem of one 2-dimensional Riemannian or pseudo-Riemannian manifold into either Euclidean space $\B R^3$
or Minkowski space $\B R^{1,2}$. The exterior differential system $\C I$ in question is discussed in detail in \cite{BC3G} (see Example 3.8 in Chapter 3).

Despite the fact that isometric embedding  is a venerable subject in differential geometry, there remain a great many basic questions and these have been well documented; see for instance  \cite{Han}. In relation to the {\it integrability} of the exterior differential system for isometric embedding, there have been comparatively few studies. Notable exceptions include work of Melko and Sterling \cite{MelkoSterling}, \cite{MelkoSterling1}, Ferus and Pedit \cite{FerusPedit96} and Terng \cite{Terng97}. For instance,  in \cite{FerusPedit96} the authors  show that the differential system for the isometric embedding of space forms into space forms is completely integrable in the sense of soliton theory.

Our own approach to the integrability of the isometric embedding problem in this paper  is to use a classical notion of integrability pioneered by Darboux (see \cite{CfB2}, Chapter 7) and studied by E. Vessiot in \cite{Vessiot}. In this paper we give a classification of all $2$-metrics whose local isometric embedding system into  $\B R^3$ or $\B R^{1,2}$ is Darboux integrable.
The motivation behind the classification is to use the Darboux integrability of the embedding system in order to construct all the embeddings in explicit finite terms of arbitrary functions. We will give a detailed example in
\S \ref{Example-sec}.

Our approach to Darboux integrability is inspired by the work of Vessiot. His approach was recently generalized to arbitrary, smooth, decomposable exterior differential systems in \cite{AFV} (see \S2 for definitions).  In particular, this approach features a Lie transformation group--- the {\it Vessiot group} of $\C I$---acting as special symmetries of the characteristic distributions of $\C I$ which, in addition, preserve the foliation induced by the first integrals of the characteristics. Among other things, the isomorphism class of the Vessiot group is an invariant of the  exterior differential system $\C I$ up to contact transformations and the general theory applies to systems whose integral submanifolds have arbitrary dimension greater than 1. Of both theoretical and practical importance  is that the Vessiot group defines a {\it superposition formula} which permits one to compute all the integral manifolds of $\C I$ from the superposition of the integral manifolds of its singular systems.

In the case when the exterior differential system $\C I$ is
hyperbolic (a special case of decomposable) with 2-dimensional integral submanifolds, the integral submanifolds of the singular systems are 1-dimensional and can often be represented in terms of quadrature. In many cases of interest these systems project to a quotient manifold which is locally equivalent to some (partial) prolongation of the jet space $J^1(\B R, \B R^q)$, for some $q\geq 1$. This equivalence can be used to express their solutions in terms of arbitrary functions and a finite number of their derivatives with all quadrature eliminated. The superposition formula then provides the explicit solution  of $\C I$ in finite terms of arbitrary functions and their derivatives. See \cite{ClellandVassiliou} for a detailed illustration of this program in the context of harmonic maps and \cite{AFV} for examples in which the dimension of integral manifolds is greater than two. For a textbook account of this material we refer to Chapter 10 in \cite{CfB2}.

It is therefore significant when a geometric problem exhibits Darboux integrability. 
For instance, it is well known that a minimal surface without umbilic points in Euclidean space $\B E^3$ admits isothermal coordinates $\xi,\eta$ such that the function $v(\xi,\eta)$ featured in its induced metric $e^{2v}(d{\xi}^2+d{\eta}^2)$ satisfies the elliptic Liouville equation
\begin{equation}\label{E_liouville}
v_{\xi\xi}+v_{\eta\eta}=e^{-2v}
\end{equation}
(see, e.g., \cite{Brito}).  Even though this equation
is elliptic, it turns out that one can still use the complex singular systems to view it as a Darboux integrable system.
Following the procedure in \cite{AFV} one can obtain the well known general solution
$$
v=-\ln\left(\frac{2|{f}'(z)|}{1+|{f}(z)|^2}\right)
$$
of (\ref{E_liouville}) depending upon an arbitrary holomorphic function ${f}$ of the complex variable $z=\xi+ i\eta$. See \cite{Brito} for an exposition of how this solution $v$ can be used in the construction of minimal surfaces in
$\B R^3$.

Another example is provided by surfaces of mean curvature 1 in hyperbolic space. Here too the differential system for such surfaces is Darboux integrable and there is a Weierstrass representation  (\cite{BryantSurfaces}, Theorem A) for so-called {\it Bryant surfaces}.  Interestingly, as pointed out in \cite{BryantSurfaces}, there is no such representation for constant mean curvature surfaces (CMC) in positive curvature space forms. However,  CMC surfaces in the 3-sphere can be constructed using loop groups \cite{DPW}. Numerous other applications of Darboux integrability in geometric problems can be listed.

The outline of this paper is as follows. In \S 2, we give a brief introduction to the theory of Darboux integrability, the Vessiot group, and the superposition formula.
In \S \ref{ImmersionsToE3}, we introduce the isometric embedding system for Riemannian 2-metrics into Euclidean $\B R^3$ and classify those metrics for which this system is Darboux integrable. It turns out that all such metrics admit a Killing field, and we exploit this symmetry to construct explicit normal forms.
In \S \ref{LorentzianSurfaces}, we perform a similar analysis for embedding Lorentzian 2-metrics into Minkowski space $\B R^{1,2}$.
Finally, in \S \ref{Example-sec}, we study one of the metrics from our classification in detail and use the tools from \cite{AFV} to construct all its embeddings in finite terms of arbitrary functions. We also show that the geometric Cauchy problem for this metric reduces to a system of two first-order ordinary differential equations for two unknown functions---or equivalently, to a single second-order scalar ODE---and we identify a large class of initial data for which the problem reduces to quadrature. We then give a simple explicit example to illustrate this construction.

\section{Darboux Integrability}\label{Darboux-int-sec}

As the name implies, the notion of Darboux integrability originated in the 19th century with Darboux, and it was most significantly developed by Goursat \cite{GoursatBook}. Classically, it was a method for constructing the ``general solution" of a second order PDE in 1 dependent and 2 independent variables
\[
F(x,y,u,u_x,u_y, u_{xx}, u_{xy}, u_{yy})=0
\]
that generalised the so-called ``method of Monge." It relies on the notions of {\it characteristics} and {\it first integrals}. We refer the reader to \cite{GoursatBook},  \cite{CfB2}, \cite{Vassiliou2000}, \cite{Vassiliou2001} for further information on classical Darboux integrability.

In this paper, we use a new geometric formulation of Darboux integrable exterior differential systems \cite{AFV}. At the heart of this theory are the fundamental notions of a {\it Vessiot group} and the {\it superposition formula}, which are our main tools for the study of the isometric embedding system.

For simplicity of exposition, we shall describe the new geometric formulation in the context of a commonly studied special case, namely,  semilinear systems of partial differential equations (PDE) in two independent variables
\begin{equation}\label{semilinear}
\bsy{u}_{xy}=\bsy{f}(x,y,\bsy{u}, \bsy{u}_x, \bsy{u}_y),
\end{equation}
where $\bsy{u}$ and $\bsy{f}$ are $\B R^n$-valued. Each solution possesses a double foliation by curves called {\it  characteristics}.
Such PDE often model wave-like phenomena,
and projection of these curves into the independent variable space describes the space-time history of the wave propagation.
The characteristics of $\bsy{u}_{xy}=\bsy{f}$ are the integral curves of a pair of rank $n+1$ distributions
\begin{equation}\label{characheristics}
H_1=\left\{D_x+D_y\bsy{f}\cdot\P {\bsy{u}_{yy}},\ \ \P {\bsy{u}_{xx}}\right\},\ \ \ H_2=\left\{D_y+D_x\bsy{f}\cdot\P {\bsy{u}_{xx}},\ \P {\bsy{u}_{yy}}\right\},
\end{equation}
on the PDE submanifold $\bsy{R}\subset J^2(\B R^2,\B R^n)$ defined by $\bsy{u}_{xy}=\bsy{f}$ where
\[
D_x=\P x+\bsy{u}_x\cdot\P {\bsy{u}}+\bsy{u}_{xx}\cdot\P {\bsy{u}_x}+\bsy{f}\cdot\P {\bsy{u}_y},\ \ \ D_y=\P y+\bsy{u}_y\cdot\P {\bsy{u}}+\bsy{f}\cdot\P {\bsy{u}_x}+\bsy{u}_{yy}\cdot\P {\bsy{u}_y}
\]
are the {\it total differential operators} along solutions of the PDE. The notation here means, for instance,
$$
\bsy{u}_{xx}\cdot\P {\bsy{u}_x}=\sum_{i=1}^nu^i_{xx}\P {u^i_x}.
$$
Note that if $\CV$ is the pullback to $\bsy{R}$ of the contact system on $J^2(\B R^2,\B R^n)$ then
$H_1\oplus H_2=\ann V$.
(When $V$ is a sub-bundle of the cotangent bundle $T^*M$ of manifold $M$, we will often abuse notation
and refer to $V$ as a Pfaffian system.  When a distinction is necessary, we will let
the corresponding letter $\CV$ denote the exterior differential system generated
differentially by sections of $V$.)

\begin{defn}
If $D$ is a distribution on a manifold $M$, then a function $h:M\to\mathbb{R}$ is said to be a {\it first integral} of $D$ if $Xh=0$ for all $X\in D$.
Equivalently, for $V = \ann D \subset T^*M$, we say $h$ is a first integral of $V$ if $dh$ is a section of $V$.
\end{defn}

For later use, we will briefly review the key results we require from the theory of Darboux integrable exterior differential systems, \cite{AFV}.

\begin{defn}\label{defnDecomposable}
An exterior differential system $\CI$ on $M$ is said to be {\it decomposable of type} $[p,q]$, for $p,q\geq 2$, if about each point $x\in M$, there is a coframe
\begin{equation}\label{decomposableCoframe}
\th^1,\ldots,\th^r, \hat{\s}^1,\ldots,\hat{\s}^p, \check{\s}^1,\ldots,\check{\s}^q
\end{equation}
such that $\CI$ is algebraically generated by the 1-forms and 2-forms
\begin{equation}\label{decomposableEDS}
\CI=\{\th^1,\ldots,\th^r,\hat{\O}^1,\ldots,\hat{\O}^s,\check{\O}^1,\ldots,\check{\O}^t\},
\end{equation}
where $s,t\geq 1$, $\hat{\O}^b\in \O^2( \hat{\s}^1,\ldots,\hat{\s}^p)$ and
$\check{\O}^\b\in\O^2( \check{\s}^1,\ldots,\check{\s}^q)$.
The differential systems algebraically generated by
\begin{equation}\label{singularSystems}
\hat{\CV}=\{\th^i,\hat{\s}^a,\check{\O}^\b\},\ \ \check{\CV}=\{\th^i,\check{\s}^\a,\hat{\O}^b\},
\end{equation}
are said to be the associated {\it singular differential systems} for $\CI$ with respect to the decomposition
\eqref{decomposableEDS}.  The distributions annihilated by $\{\theta^i, \hat{\s}^a\}$ and
$\{\theta^i, \check{\s}^\alpha\}$ are referred to as the associated {\em characteristic distributions}.
\end{defn}

For a sub-bundle $V\subset T^*M$, let us denote by $V^\infty$ the final element of its derived flag.

\begin{defn}\label{defnDarbouxPair}
Let $\hat{V}, \check{V}$ be a pair of Pfaffian systems on a manifold $M$, such that
\begin{enumerate}
\item[a)] $V_1+V_2^{\infty}=T^*M\ \ \text{and}\ \ V_1+V_2^{\infty}=T^*M$;
\item[b)] $V_1^{\infty}\cap V_2^{\infty}=\{0\}$;
\item[c)] $d\o\in\O^2(\hat{V})+\O^2(\check{V})\ \ \ \forall\ \o\in\O^1(\hat{V}\cap\check{V})$
\end{enumerate}
Then $\{\hat{V}, \check{V}\}$ is said to be a {\it Darboux pair}.
\end{defn}

\begin{defn}\label{DarbouxIntegrableEDS}
Let $\CI$ be a decomposable differential system and suppose that the associated singular systems $\hat{\CV}, \check{\CV}$ are Pfaffian. Then $\CI$ is said to be {\it Darboux integrable} if $\{\hat{V}, \check{V}\}$ determine a Darboux pair.
\end{defn}

Condition b) of Definition \ref{defnDarbouxPair} implies that there are no first integrals that are common to $\hat{V}$ and $\check{V}$, while condition a) implies that each of $\hat{V}$ and $\check{V}$ possess {\it sufficiently many} first integrals.    For differential systems with two independent variables such as (\ref{semilinear}) we have
\begin{lemma}\label{DI_for_semilinearSystems}
A semilinear system $\bsy{u}_{xy}=\bsy{f}$ with $\bsy{u, f}\in \mathbb{R}^n$ is Darboux integrable at order $2$ if and only if each of its characteristic distributions $H_i$ has at least $n+1$ independent first integrals.
\end{lemma}
\begin{proof}
The semilinear system of the Lemma statement defines a submanifold $\bsy{R}$ of dimension $5n+2$ inside the jet space $J^2(\B R^2,\B R^n)$ whose dimension is $6n+2$. The pullback to $\bsy{R}$ of the contact system on $J^2(\B R^2,\B R^n)$ is the sub-bundle $V\subset T^*J^2(\B R^2,\B R^n)$ spanned by $3n$ 1-forms $\th_0^i,\ \th_1^i,\ \th_2^i$, $1\leq i\leq n$, satisfying structure equations
$$
\begin{aligned}
&d\th_0^i\equiv 0,\cr
&d\th_1^i\equiv \pi^i_1\wedge\o^1,\cr
&d\th_2^i\equiv \pi^i_2\wedge\o^2,
\end{aligned}\mod V,
$$
where $\o^1=dx, \o^2=dy$ are the independence forms, $\pi^i_1\equiv du^i_{xx}\mod \{\o^1,\o^2\}$ and $\pi^i_2\equiv du^i_{yy}\mod \{\o^1,\o^2\}$. The associated singular differential systems are Pfaffian with degree 1 components
$$
\hat{V}=\{\th_0^i,\th_1^i,\th_2^1,\pi_1^i,\o^1\}\ \text{and}\ \check{V}=\{\th_0^i,\th_1^i,\th_2^1,\pi_2^i,\o^2\},
$$
each of rank $4n+1$. Condition a) of Definition \ref{defnDarbouxPair} is satisfied if
$\dim\bsy{R}=5n+2\leq \rk \hat{V}+\rk\check{V}^\infty=4n+1+\rk\check{V}^\infty$
and  $\dim\bsy{R}=5n+2\leq \rk \check{V}+\rk\hat{V}^\infty=4n+1+\rk\hat{V}^\infty$. This implies that
$$
\rk\hat{V}^\infty\geq n+1\ \ \ \text{and}\ \ \ \rk\check{V}^\infty\geq n+1.
$$
Thus, the singular systems $\{\hat{V}, \check{V}\}$ of the decomposable EDS $\CI$ whose degree 1 component is $V$ form a Darboux pair (and hence $\CI$ is Darboux integrable) if and only if each has at least $n+1$ first integrals.
\end{proof}
\begin{rem}
The formulation of Darboux integrability culminating in Definition \ref{DarbouxIntegrableEDS} generalizes the well known classical definition to encompass any decomposable exterior differential system. This includes a vast collection of systems of partial differential equations with no general constraint on the number of independent or dependent variables or the order of the system. Since the definition is expressed in terms of EDS it can be applied to differential equations on manifolds. The purpose of Lemma \ref{DI_for_semilinearSystems} is to show how it applies to well known examples that are current in the literature. These ideas will be used in our treatment of the isometric embedding system in later sections of the paper.
\end{rem}

A key theorem, proven in \cite{AFV}, is a result in the inverse problem in the theory of quotients:
\begin{thm}\label{mainAFVthm}
Let $(M, \CI)$ be a Darboux integrable Pfaffian system. Then there are Pfaffian systems $(\wh{M}_1, \wh{\CW}_1)$, $(\wh{M}_2, \wh{\CW}_2)$ which admit a common Lie group $G$ of symmetries such that:
\begin{enumerate}
\item The manifold $M$ can be locally identified as the quotient of $\wh{M}_1\times \wh{M}_2$ by a diagonal action of $G$;
\item We have the identification
$$
\CI=\left(\pi_1^*\wh{\CW}_1+\pi_2^*\wh{\CW}_2\right)\Big/G,
$$
where $\pi_i:\wh{M}_1\times \wh{M}_2\to \wh{M}_i,\ i=1,2$ are the canonical projection maps;
\item The quotient $\bsy{\pi} :\wh{M}_1\times \wh{M}_2\to M$ by the diagonal $G$-action defines a surjective superposition formula for $(M,\CI)$.
\end{enumerate}
\end{thm}
\vskip 5 pt
In this context, a {\it superposition formula} for $(M,\CI)$ is a map $\bsy{\pi} :\wh{M}_1\times \wh{M}_2\to M$ such that if $\sigma_i:\mathcal{U}_i\to \wh{M}_i$ are integral submanifolds of $(\wh{M}_i,\wh{\CW}_i)$, then $\bsy{\pi}\circ(\sigma_1,\sigma_2):\mathcal{U}_1\times\mathcal{U}_2\to M$ is an integral submanifold of $\CI$. A superposition formula is {\it surjective} if every solution of $\CI$ can be expressed in this form for a fixed superposition formula $\bsy{\pi}$  as $\sigma_i$ range over the integral manifolds of $\wh{\CW}_i$.
\vskip 5 pt
Reference \cite{AFV} is devoted to a proof of this theorem and to the identification and explicit construction of all the entities mentioned there, including the Lie group of symmetries $G$, known as the {\it Vessiot group} of the Darboux integrable exterior differential system $\CI$.  It is proven that the isomorphism class of the Vessiot group is a diffeomorphism invariant of such systems. For applications of the above theory of Darboux integrability we refer, for instance to \cite{AF1}, \cite{AF2}, \cite{Nie} and \cite{ClellandVassiliou}.

In the present paper we study the EDS for the isometric immersion of Riemannian and pseudo-Riemannian 2-metrics into either Euclidean or Minkowski spaces of dimension 3. As we shall see in that case, each EDS will be Darboux integrable if and only if the characteristic distributions each have at least 2 independent first integrals.

\section{Immersions of Riemannian surfaces into $\mathbb{R}^3$}\label{ImmersionsToE3}
In this section, we classify all Riemannian 2-metrics such that the EDS for isometric embedding into Euclidean space $\B R^3$ is Darboux integrable.

\subsection{The Isometric Embedding System}
Let $\FR$ be the orthonormal frame bundle of Euclidean $\R^3$, which carries canonical 1-forms $\w^i$ and connection 1-forms $\w^i_j$ (with $\w^i_j = -\w^j_i$), where $1\le i,j \le3$.
We let $\brho$ denote the projection from $\FR$ to $\R^3$, and $\ve_i$ the vector-valued components of the frame, and note that
the canonical and connection forms are defined as components of the exterior derivatives of these vector-valued functions:
\begin{equation}\label{d-of-basepoint} d\brho = \ve_i \w^i \end{equation}
and
\begin{equation}\label{d-of-frame} d\ve_i = \ve_j \w^j_i. \end{equation}
Differentiating these equations yields the usual structure equations for $\FR$:
\begin{equation}\label{R3streq}
d\w^i = -\w^i_j \wedge \w^j, \quad d\w^i_j = -\w^i_k \wedge \w^k_j, \qquad 1 \le i,j,k \le 3.
\end{equation}

Let $M$ be a connected, oriented surface with Riemannian metric $\mg$, and let $\FM$ be the oriented orthonormal frame bundle of $M$, with
projection $\pi: \FM \to M$.
This bundle carries canonical 1-forms $\eta^1, \eta^2$ and connection form $\eta^1_2$, satisfying structure equations
\begin{equation}\label{FMstreq}
d\eta^1 = -\eta^1_2 \wedge \eta^2, \quad d\eta^2 = -\eta^2_1 \wedge \eta^1, \quad d\eta^1_2 = K \eta^1 \wedge \eta^2,
\end{equation}
where $\eta^2_1 = -\eta^1_2$ and $K$ is the Gauss curvature of the metric.  While the $\eta^i$ are 1-forms on $\FM$, the quadratic differential
$
\left(\eta^1\right)^2+\left(\eta^2\right)^2
$
is well-defined on $M$ and coincides with $\mg$.
The canonical forms are sometimes called `dual' 1-forms
since, given any (local) section $f$ of $\FM$, the 1-forms $f^* \eta^1$ and $f^*\eta^2$
are dual to the component vector fields $\vv_1, \vv_2$ of the framing.

On $\FM \times \FR$, we define a Pfaffian system $\I$ generated by the 1-forms
$$\theta_0 := \w^3, \quad \theta_1 := \w^1 - \eta^1, \quad \theta_2 := \w^2 - \eta^2,\quad \theta_3 := \w^1_2 - \eta^1_2.$$
Here, the canonical and connection 1-forms on $\FM$ and $\FR$ are pulled back to the product of these spaces, but we suppress
the pullback notation; similarly, we extend $\pi$ and $\rho$ to the product space by composing with
maps to each factor.
We will only consider integral surfaces $S$ of system $\I$ that satisfy the {\em independence condition}
$\eta^1 \wedge \eta^2 \ne 0$; this is enough to guarantee that $\pi\vert_S$ is a local diffeomorphism and
hence a covering map from $S$ to an open subset of $M$.  We have the following basic result:

\begin{prop}  Let $S$ be an integral surface of $\I$ such that $\pi\vert_S$ is a diffeomorphism
onto an open subset $U \subset M$, and let
$\sigma:U \to S$ be its inverse; then $\psi = \rho \circ \sigma: U \to \R^3$ is an isometric immersion.
Conversely, if $\psi: U \to \R^3$ is an isometric immersion and $f: U \to \FM$ is an orthonormal framing
defined on $U$, then there is lift $\sigma$ of $f$ into $\FM \times \FR$ whose image is an integral surface $S$
of $\I$, and $\rho\vert_S = \psi \circ \pi\vert_S$.
\end{prop}

\begin{center}
\begin{tikzpicture}
\node (surf) at (0,0) {$U$};  
\node (r3) at (2,0) {$\R^3$};
\node (prod) at (1,1.5) {$\FM \times \FR$};
\node (fm) at (-1,1) {$\FM$};
\draw[->,thick] (surf) -- node[below] {$\psi$} (r3);
\draw[->,thick] (prod.315) -- node[above right] {$\rho$} (r3);
\draw[->,thick,shorten >= -.5ex] (prod.west) -- (fm);
\draw[->,thick] (prod.225) -- node[above left] {$\pi$} (surf);
\draw[->,thick] (fm) -- (surf);
\draw[->,dashed,bend right] (surf) to [out=-30,in=-160] node[below right] {$\sigma$} (prod);
\draw[->,dashed,bend left] (surf) to [out=30,in=150] node[below left] {$f$} (fm);
\end{tikzpicture}
\end{center}

\begin{proof} We begin by establishing the first statement.
The vectors $\ebar_i = \ve_i \circ \sigma$ give an
orthonormal framing along the image of $\psi$.  Taking the dot product of the differential
of our mapping $\psi$ with a frame vector gives
\begin{equation}\label{ebardot}
\ebar_i \cdot d\psi = ( \ve_i \circ \sigma)\cdot d(\rho \circ \sigma)
                            = (\ve_i \cdot d\rho ) \circ \sigma \\
                            = \sigma^* \w^i,
\end{equation}
where the last equality follows from the defining equation \eqref{d-of-basepoint} of the canonical forms.
Because $\sigma^* \w^3=0$, then $\ebar_3 \cdot \psi_*(\vw) = 0$ for any vector $\vw$ tangent to $U$.  Thus,
the tangent space to the image of $\psi$ lies in the span of $\ebar_1, \ebar_2$, and we next calculate the differential
of $\psi$ in terms of these vectors.

The 1-forms $\overline{\eta}^m = \sigma^* \eta^m$ for $m=1,2$ give an orthonormal coframe field on $U$.  Let $\vv_1, \vv_2$ be the dual framing.
Because $\sigma^* \w^m= \overline{\eta}^m$, the computation \eqref{ebardot} implies that $\ebar_m \cdot \psi_*(\vw) = \overline{\eta}^m(\vw)$
for any tangent vector $\vw \in TU$.
Hence $\psi_*(\vv_m) = \ebar_m$, and $\psi$ is an isometric immersion.

To establish the converse, let $\vv_m$ for $m=1,2$ be the members of the orthonormal frame field $f$ on $U$,
let $\ebar_m = \psi_* \vv_m$, and $\ebar_3 = \ebar_1 \times \ebar_2$.   For $p\in U$ let
$$\sigma:p \mapsto (p,\vv_1, \vv_2; \psi(p),\ebar_1, \ebar_2, \ebar_3) \in \FM \times \FR$$
and let $S=\sigma(U)$.
Then \eqref{ebardot} implies that $\sigma^* \w^3 = 0$ and $\sigma^*(\w^m - \eta^m) = 0$ for $m=1,2$.
Finally, since
\begin{equation}\label{abovecomp}
0 = \sigma^* d(\w^1 - \eta^1) = -\sigma^*(\w^1_2 \wedge \eta^2 + \w^1_3 \wedge \w^3 -\eta^1_2 \wedge \eta^2) = -\sigma^*(\w^1_2 - \eta^1_2) \wedge \overline{\eta}^2
\end{equation}
and similarly $0 = \sigma^*d (\w^2 - \eta^2) = \sigma^*(\w^1_2 - \eta^1_2) \wedge \overline{\eta}^1$,
it follows from the independence condition that $\sigma^*(\w^1_2 -\eta^1_2)=0$.  Hence, $S$ is an integral surface of $\I$,
and the equation of maps follows from $\rho\circ \sigma = \psi$.
\end{proof}

%
%

Next we will calculate the algebraic generators and singular systems of $\I$.
The calculation \eqref{abovecomp} implies that $d\theta_1 \equiv 0 $ and $d\theta_2\equiv 0$ modulo the
1-forms of $\I$, so that $\CI$ is generated algebraically by $\th_0, \th_1, \th_2, \th_3$ together with the 2-forms
$$ \Omega_0 := \w^3_1 \wedge \eta^1 + \w^3_2 \wedge \eta^2,\qquad
\Omega_1 := \w^3_1 \wedge \w^3_2 - K \eta^1 \wedge \eta^2
$$
which satisfy $d\theta_0 \equiv -\Omega_0$ and $d\theta_3 \equiv \Omega_1$ modulo the 1-forms.

If there are two linearly independent combinations of these that are decomposable, then the factors of each 2-form, together with the generator 1-forms of $\I$, span one of the singular differential systems of $\I$ (see Definition 2 above).  One computes
$$(a \Omega_1 + b \Omega_0) \wedge (a \Omega_1 + b \Omega_0) = 2(a^2 K + b^2)\, \w^3_1 \wedge \w^3_2 \wedge \eta^1 \wedge \eta^2.$$
Thus, decomposable combinations are given by taking $b/a = \pm \sqrt{-K}$.  In order that
the singular systems have smoothly defined 1-forms, we will restrict to either
the {\em hyperbolic case} $K <0$ or the {\em elliptic case} $K >0$. We will say the EDS $\CI$ is
{\it hyperbolic} or {\it elliptic} according to whether the Gauss curvature $K$ is negative or positive, respectively.
In either case, we introduce a smooth positive function $k$ on $M$ such that $k^2 = |K|$.
For future purposes, we introduce the components of the first and second covariant derivatives of $k$; these are the functions
$k_i$ and $k_{ij}$ on $\FM$ satisfying
\begin{equation}\label{k-covariant-derivatives}
dk = k_i \eta^i, \qquad dk_i = k_j \eta^j_i + k_{ij} \eta^j,
\end{equation}
where we now take $1\le i,j \le 2$.

\begin{defn}
If $(M,\bsy{g})$ is a 2-dimensional Riemannian or pseudo-Riemannian manifold whose isometric embedding system
into $\B R^3$ or $\B R^{1,2}$ is Darboux integrable, then we will say that the metric $\bsy{g}$  itself is {\it Darboux integrable} with respect to their embedding.
\end{defn}
\begin{rem}
We will see that the requirement that $(M,\bsy{g})$ be Darboux integrable with respect to an embedding space is that the singular systems $(\hat{\CV},\check{\CV})$ of the corresponding isometric embedding system each have two independent first integrals.
\end{rem}

\subsection{Integrability conditions for the hyperbolic case}\label{hypersec}
The purpose of this section is to prove the following.

\begin{thm}\label{IcondsRiemHyp}
Let $\bsy{g}$ be a Riemannian metric on $M$, and suppose that the Gauss curvature $K$ of $\bsy{g}$ is negative, so that $\CI$ is hyperbolic. Let
$q=k^{-3/2}=(-K)^{-3/4}$. Then $\CI$ is Darboux integrable if and only if the function $q$ on $M$ satisfies the differential equations
$$
q_{11}=q_{22}=-3q^{-1/3},\qquad q_{12}=q_{21}=0,
$$
where the functions $q_{ij}$ on $\FM$ are defined, similarly to \eqref{k-covariant-derivatives}, by the covariant equations
\begin{subequations}\label{q-covariant-derivatives}
\begin{align}
dq&=q_i\eta^i,\label{dq-first} \\
dq_i &= q_j \eta^j_i + q_{ij}\eta^j. \label{dq-second}
\end{align}
\end{subequations}
\end{thm}
\begin{proof}
In this case we compute that
$$\Omega_1 \pm k\Omega_0 = (\w^3_1 \mp k \eta^2) \wedge (\w^3_2 \pm k \eta^1).$$
Thus, the two singular Pfaffian systems are
$$\pV_\pm = \{\theta_0, \theta_1, \theta_2, \theta_3, \w^3_1 \mp k\eta^2 , \w^3_2 \pm k \eta^1 \},$$
and $\pV_+, \pV_-$ are a Darboux pair for $\I$.

To determine conditions on the metric $\bsy{g}$ such that $\CI$ is Darboux integrable, we need to determine when $V_\pm$ each have at least 2 independent first integrals.
%
%
That is, the hyperbolic EDS $\I$ is Darboux-integrable if and only if its singular systems each contain a Frobenius system which has rank at least 2,
and which is transverse to the 1-forms of $\I$.   The derived flags of $V_\pm$ must terminate in these Frobenius systems (if they exist), so we will
compute the derived flag of each singular system; we begin with $V_+$.
Direct computation yields the first derived system
$$\pV_+^{(1)} = \left\{ \theta_1,\, \theta_2,\, \theta_3 - k \theta_0,\, \w^3_1 - k\eta^2 +\dfrac{k_1}{2k}\theta_0,\, \w^3_2 + k \eta^1 + \dfrac{k_2}{2k}\theta_0\right\}$$
and second derived system
\begin{multline*}
\pV_+^{(2)} = \{ \theta_)3-k\theta_0 + \dfrac{k_2}{2k} \theta_1 - \dfrac{k_1}{2k} \theta_2, \\
                \w^3_1 - k\eta^2 + \dfrac{k_1}{2k} \theta_0 + \dfrac{3k_1 k_2 - 2k_{12}k}{8k^3} \theta_1 - \dfrac{3k_1^2 -2k_{11}k + 4k^4}{8k^3} \theta_2, \\
                \w^3_2 + k\eta^1  + \dfrac{k_2}{2k} \theta_0  + \dfrac{3k_2^2 - 2k_{22}k + 4k^4}{8k^3}\theta_1 - \dfrac{3k_1k_2 - 2k_{12}k}{8k^3}\theta_2\}.
\end{multline*}
In order for $\pV_+^{\infty}$ to be of rank at least 2, $\pV_+^{(3)}$ must have rank at least 2, and furthermore it must be integrable.  (It has rank 3 if and only if $\pV_+^{(2)}$ is Frobenius, which never happens when $K \neq 0$, so in fact it must have rank exactly 2.)
%
%
Imposing the constraint that $\pV_+^{(3)}$ be rank 2, together with the corresponding condition
for $\pV_-^{(3)}$, leads to the following constraints on the Gauss curvature of $\bsy{g}$: either
\begin{equation}\label{khypercond}
k_{11} = 2k^3 + \dfrac52 \dfrac{k_1^2}{k}, \quad k_{12} = \dfrac52 \dfrac{k_1k_2}{k}, \text{ and } k_{22} = 2k^3 +
\dfrac52 \dfrac{k_2^2}{k},
\end{equation}
or
\begin{equation}\label{strawcond}
k_{11} = 2k^3 + \dfrac32 \dfrac{k_1^2}{k}-\dfrac{k_2^2}{k}, \quad k_{12} = \dfrac52 \dfrac{k_1k_2}{k}, \text{ and } k_{22} = 2k^3 +
\dfrac32 \dfrac{k_2^2}{k}-\dfrac{k_1^2}{k},
\end{equation}
where the $k_{i}$ and $k_{ij}$ are defined as in \eqref{k-covariant-derivatives}.
Either set of equations completely determines the second derivatives of $k$, and taking further covariant derivatives and equating
mixed partials generates compatibility conditions that must be satisfied if such functions are to exist.
The compatibility conditions derived from the first set \eqref{khypercond} are implied by the equations \eqref{khypercond} themselves, but those derived from the second set hold only if $k_1^2 + k_2^2 = 8k^4$.  When this condition is differentiated in turn, the results are
inconsistent with \eqref{strawcond}.  On the other hand, when we assume that \eqref{khypercond} holds, it is easy
to check that
$$\pV_+^{(3)} = \left\{\w^3_1 -k \eta^2 +\dfrac{k_1}{2k^2}(\omega^1_2 - \eta^1_2),\, \w^3_2 + k\eta^1 +\dfrac{k_2}{2k^2}  (\w^1_2 - \eta^1_2)\right\}$$
and this is a Frobenius system.  Similarly, assuming \eqref{khypercond} gives
$$\pV_-^{(3)} = \left\{\w^3_1 +k \eta^2 -\dfrac{k_1}{2k^2}(\w^1_2 - \eta^1_2),\, \w^3_2 - k\eta^1 -\dfrac{k_2}{2k^2}  (\w^1_2 - \eta^1_2)\right\}$$
which also turns out to be a Frobenius system.  Thus, \eqref{khypercond} is necessary and sufficient for Darboux integrability.

The conditions \eqref{khypercond} are slightly nicer when expressed in terms of $q$.  Indeed, one finds that \eqref{khypercond}
is equivalent to the second covariant derivatives
of $q$ satisfying
$$q_{11} = q_{22} = -3 q^{-1/3},\qquad q_{12}=q_{21}=0,$$
as required.
\end{proof}


\subsection{Integrability conditions for the elliptic case}
Now assume that we have metric of strictly positive Gauss curvature $K$ on $M$, and set $k=\sqrt{K}$. We prove a theorem similar to Theorem \ref{IcondsRiemHyp}.
\begin{thm}\label{IcondsRiemEll}
Let $\bsy{g}$ be a Riemannian metric on $M$, and suppose that the Gauss curvature $K$ of $\bsy{g}$ is positive, so that $\CI$ is elliptic. Let
$q=k^{-3/2}=K^{-3/4}$. Then $\CI$ is Darboux integrable if and only if the function $q$ on $\FM$ satisfies the differential equations
$$
q_{11}=q_{22}=3q^{-1/3},\qquad q_{12}=q_{21}=0,
$$
where the $q_{ij}$ are defined by the covariant equations \eqref{q-covariant-derivatives}.
\end{thm}

\begin{proof}
We sketch the proof, which is similar to the one for Theorem   \ref{IcondsRiemHyp}. In this case, the 2-forms $\Omega_0$ and $\Omega_1$ can be linearly combined to create decomposable generators, but only if
we use complex coefficients:
$$\Omega_1 \pm \ri k \Omega_0 = (\w^3_1 \mp \ri k \eta^2) \wedge (\w^3_2 \pm \ri k \eta^1).$$
Matters being so, we define a singular system
$$\CW := \left\{ \theta_0, \theta_1, \theta_2, \theta_3, \w^3_1 - \ri k \eta^2, \w^3_2 + \ri k \eta^1\right\},$$
which spans a sub-bundle of the complexified cotangent bundle; the other singular system
is its complex conjugate $\overline{\CW}$.  When $q$ satisfies the above conditions, $\CW$ and $\overline{\CW}$ form a Darboux pair for $\I$,
if we suitably extend the definition to encompass complex-valued 1-forms.

For Darboux integrability, it is sufficient that
$\M$ contain a Frobenius system of rank at least 2.  (This
would automatically imply that $\overline{\CW}$ also contains a Frobenius system of the same rank.)
By a similar calculation to that described in \S $\ref{hypersec}$, $\CW^{(3)}$ has rank at least two only if $k$ satisfies
\begin{equation}\label{kellipcond}
k_{11} = -2k^3 + \dfrac52 \dfrac{k_1^2}{k}, \quad k_{12} = \dfrac52 \dfrac{k_1k_2}{k}, \quad k_{22} = -2k^3 +
\dfrac52 \dfrac{k_2^2}{k}.
\end{equation}
Notice that this differs from \eqref{khypercond} only by a sign change in the expressions for $k_{11}$ and $k_{22}$.
Moreover, when \eqref{kellipcond} holds, $\CW^{(3)}$ and $\overline\CW^{(3)}$ are rank 2 and Frobenius; for example,
$$\CW^{(3)} = \left\{ \w^3_1 - \ri k  \eta^2 - \dfrac{\ri k_1}{2k^2}(\w^1_2 - \eta^1_2),
            \w^3_2 + \ri k \eta^1 - \dfrac{\ri k_2}{2k^2}(\w^1_2 - \eta^1_2)\right\}.$$
If we again set $q=k^{-3/2}$, then \eqref{kellipcond} is equivalent to
$q_{11} = q_{22}=3q^{-1/3}$ and $q_{12}=q_{21}=0$, as required.
\end{proof}

To treat both the elliptic and hyperbolic cases at
the same time, let $\eps=1$ represent the elliptic case ($K>0$) and
$\eps=-1$ the hyperbolic case ($K<0$).
Then the Darboux integrability condition is that the function $q = (\eps K)^{-3/4}$ on $\FM$ satisfy
\begin{equation}\label{qcond}
q_{11}=q_{22} = 3 \eps q^{-1/3}, \qquad q_{12}=q_{21}=0.
\end{equation}

\begin{rem}\label{no-constant}
The equations (\ref{qcond}) clearly admit no constant solutions $q$. Thus, the isometric embedding systems for the constant curvature surfaces $\mathbb{S}^2$ and $\mathbb{H}^2$ have isometric embedding systems which are not Darboux integrable.
\end{rem}

\begin{rem}The equations \eqref{qcond} are invariant under rotations of the orthonormal frame on $M$.  That is, if they hold at one
point in the fiber of $\FM$, then they hold at all points in that fiber.  In the sections that follow, we will often work with a
specific choice of orthonormal frame on $M$, and pull back the structure equations \eqref{FMstreq}, as well as the defining equations
\eqref{q-covariant-derivatives} for the $q_i$ and $q_{ij}$, via the corresponding section $f:M \to \FM$.  When we pull back the EDS $\I$ via this section, we obtain a Pfaffian system on $M \times \FR$, where the $\eta^m$ are replaced by their pullbacks $\overline{\eta}^i = f^*\eta^i$.
However, for ease of notation in what follows we will omit the bars on the $\eta^i$ and $\eta^1_2$.
\end{rem}

\subsection{Normal forms for Darboux-integrable metrics}\label{normal-forms-Riemannian-sec}
In this section we will determine the metrics $\mg$ for which the isometric embedding system is Darboux-integrable.
It follows from Remark \ref{no-constant} that such metrics cannot have constant Gauss curvature, and therefore admit at most
one Killing field.  In fact, determining this set of metrics is made easy by the fact that the integrability conditions
imply that $\mg$ has a Killing field.  Once we choose local coordinates on $M$ that are adapted to this Killing field,
the integrability conditions are reduced to a single second-order ODE, as we will now show.

\begin{thm}\label{existence-of-Killing}
Suppose that $\mg$ is a Riemannian metric on $M$ for which the embedding EDS $\CI$ is Darboux integrable.  Let $J$ be the complex structure on $M$ compatible with $\mg$ and the orientation.  Then $\vV = J \nabla q$ is a Killing field.
Moreover, near any point there exist local coordinates $(s,t)$ on $M$ with respect to which
$$
\bsy{g}=ds^2+q'(s)^2 dt^2, \qquad\text{and}\quad \vV = \dfrac{\partial}{\partial t},
$$
where $q(s)$ is a strictly monotone solution of the ODE
\begin{equation}\label{qode}
q'' = 3\eps q^{-1/3}.
\end{equation}
\end{thm}

\begin{proof} Let $\vv_1, \vv_2$ be an arbitrary oriented orthonormal framing on $M$, and let $\eta^1, \eta^2$ be the dual
coframe field.  Then $\nabla q = q_1 \vv_1 + q_2 \vv_2$ and $\vV = J \nabla q = q_1 \vv_2 - q_2 \vv_1.$
Recall that the condition that a vector field $\vV$ is Killing is that $\Lie_{\vV} \mg = 0$, and this is equivalent to
the condition that the 1-form $\chip = \vV^\flat$ has a skew-symmetric covariant derivative.  In this case, if we let
$\chip = x_1 \eta^1 + x_2 \eta^2$, then $x_1 = -q_2$ and $x_2=q_1$.  Then if we define the components $x_{ij}$ of the covariant derivative $\nabla \chip$
as usual by the equations
$$d x_i = x_j \eta^j_i + x_{ij} \eta^j$$
analogous to \eqref{k-covariant-derivatives} and \eqref{q-covariant-derivatives}, it follows from \eqref{qcond} that
$$x_{12}=-x_{21} =-3\eps q^{-1/3}, \quad x_{11}=x_{22}=0.$$
Thus, $\vV$ is a Killing field.


Near any point, we now choose a special orthonormal framing $\vv_1, \vv_2$ adapted so that $\vV$ is a positive multiple of $\vv_2$, and hence $q_1>0$ and $q_2=0$.  Substituting this and \eqref{qcond} into \eqref{dq-second} gives
\begin{equation}\label{when-q2-zero}
dq_1 = 3\eps q^{-1/3} \eta^1, \qquad 0 = dq_2 = q_1 \eta^1_2 + 3\epsilon q^{-1/3} \eta^2.
\end{equation}
From the second equation it follows that $\eta^1_2 = -3\eps\dfrac{ q^{-1/3}}{q_1} \eta^2$, and therefore
$$d\eta^1 = -\eta^1_2 \wedge \eta^2 = 0.$$
So, we can choose a local coordinate $s$ such that $\eta^1 = ds$.  By appealing to the existence of flowbox coordinates, we can
choose a second coordinate $w$ so that $dw(\vV)=1$.  Since $ds(\vV)=0$, flow by $\vV$ in the $(s,w)$ coordinate system is just translation
in $w$.   Because \eqref{dq-first}
 now gives $dq = q_1 ds$, $q$ must be a function of $s$ only, $q_1 = dq/ds$, and the first equation in \eqref{when-q2-zero} implies that $q(s)$ satisfies the second-order ODE \eqref{qode}.

The dual vector fields $\di/\di s$ and $\di/\di w=\vV$ are not necessarily orthogonal; in fact, if we set
$\mu = dw(\vv_1)$ then $$\di/\di s = \vv_1 - \mu \vV.$$
We wish to replace $w$ with another coordinate $t$ such that $dt(\vV)=1$ and $dt(\vv_1)=0$, so that $(s,t)$ are orthogonal coordinates.
These conditions imply that $dt = dw - \mu \,ds$, but the right-hand side is a closed 1-form only if $\mu$ is a function of $s$.
Fortunately, this follows from the fact that
$$g(\di/\di s, \di/\di w) = -\mu g(\vV,\vV) = - \mu |\nabla q|^3 = -\mu q'(s)^2,$$
and the observation that flow by $\vV$ preserves the components of the metric in the $(s,w)$ coordinates, so that
$g(\di/\di s, \di/\di w)$ is a function of $s$ only.  We then obtain $t$ by integrating $dw - \mu(s) \,ds$.

In our modified coordinate system $(s,t)$ we have $dt(\vV)=1$ and hence $\eta^2 = q'\,dt$.  It follows that
\begin{equation}\label{gform}
\mg = (\eta^1)^2 + (\eta^2)^2 = ds^2 + (q')^2 dt^2.
\end{equation}
\end{proof}

Because $q(s)$ satisfies a second-order ODE, the family of metrics \eqref{gform} depend upon two arbitrary parameters. However, this family of metrics admits a 2-dimensional group action generated by translations in $s$ and a scaling symmetry of the form $s\mapsto \lambda s,\ q\mapsto \lambda^{3/2} q$.  (These arise from the symmetries of the ODE (\ref{qode}) for $q$.)
So, we might expect that, modulo this action, there is just a finite list of distinct metrics.  In what follows, we will take advantage of these transformations to put these metrics into a small number of possible normal forms.

\begin{thm}\label{normalf} Let $\mg$ be a Riemannian 2-metric for which the isometric embedding system is Darboux integrable.  If the system is elliptic, then there are local coordinates in which the metric takes one of the following forms:
\begin{align}
1. \quad \mg &= \cosh^4 u\, (du)^2 + \sinh^2 u\, (dv)^2, \qquad u>0;  \label{norm1}\\
2. \quad \mg &= \sinh^4 u\, (du)^2 + \cosh^2 u\, (dv)^2; \label{norm2}\\
3. \quad \mg &= u^2 \left( (du)^2 + (dv)^2 \right), \qquad u>0. \label{norm3}
\intertext{If the system is hyperbolic, then the metric takes the following form:}
4. \quad \mg &= \cos^4 u\, (du)^2 + \sin^2 u\, (dv)^2, \qquad 0 < u <\pi/2. \label{norm4}
\end{align}
\end{thm}
\begin{proof}
The starting point is integrating the ODE \eqref{qode} one time.  Multiplying \eqref{qode} by $q'$ and taking antiderivatives gives
\begin{equation}\label{intonce}
(\tfrac13 q')^2 = \epsilon q^{2/3} - C
\end{equation}
for some constant $C$.
There are several cases, depending on the signs of the constants.

\begin{enumerate}
\item[1.] Assume $\epsilon =1$ and $C = a^2 >0$.  (Here and in subsequent cases we take $a>0$.)  Then we have
$$\left( \dfrac{q'}{3a}\right)^2 = \left(\dfrac{q^{1/3}}{a}\right)^2 - 1.$$
This equation is satisfied by $q' = 3a\sinh u$ and $q^{1/3} = a\cosh u$ for $u>0$. Note that the prime
 indicates differentiation by $s$, so that taking the derivative of the second equation implies $ds/du = a^2 \cosh^2 u$.  Thus, in this case the metric may be written as
$$\mg = a^4 \cosh^4 u \,(du)^2 + 9 a^2 \sinh^2 u \,(dt)^2.$$
The metrics in this family differ only by scaling, so we may take $a=1$; then setting $v=3t$ yields the normal form \eqref{norm1}.
For this metric, $K=\sech^4 u$.  As a metric on the half-plane, this degenerates as $u\to 0$.  However,
if we take the quotient by the discrete translation $v \mapsto v +2\pi$ (under which the Killing orbits become circles),
the resulting metric closes up smoothly as $u\to 0$.

\medskip
\item[2.] Assume $\epsilon = 1$ and $C= -a^2 <0$.  Then \eqref{intonce} is satisfied by
$q' =dq/ds= 3a \cosh u$ and $q^{1/3} = a\sinh u$ for $u>0$, so that $ds/du = a^2 \sinh^2 u$.  By means similar to Case 1 we arrive at the
normal form \eqref{norm2}.
For this metric, $K = \csch^4 u$, and the metric is incomplete as $u\to 0$, as $K$ becomes unbounded.

\medskip
\item[3.] Assume $\epsilon =1$ and $C=0$.  We have $q' = 3 q^{1/3}$, a separable ODE,
with solution $q = (2s)^{3/2}$, after a suitable translation in $s$.  Using $u=\sqrt{2s}$, rescaling the metric, and letting $v$ be an appropriate
constant times $t$, we obtain the normal form \eqref{norm3}.
For this metric,
$K = u^{-4}$, and again the metric is incomplete (and $K$ becomes unbounded) as $u \to 0$.

\medskip
\item[4.] Assume $\epsilon =-1$; then $C$ is necessarily negative. Setting $C = -a^2$ in \eqref{intonce} gives
$$\left( \dfrac{q'}{3a}\right)^2 = 1 - \left(\dfrac{q^{1/3}}{a}\right)^2.$$
This is satisfied by setting $q'=dq/ds = -3a \sin u$ and $q^{1/3} = a \cos u$, so that $ds/du = a^2 \cos^2 u$.  By means similar to
Cases 1 and 2 we arrive at the normal form \eqref{norm4},
for which $K = -\sec^4 u$.  Once again, if we take the quotient by $v \mapsto v + 2\pi$,
this metric closes up smoothly as $u \to 0$.  However, it is incomplete as $u \to \pi/2$,
as $|K|$ becomes unbounded.
\end{enumerate}
\end{proof}

%
%
%
%

\begin{rem}
The metric \eqref{norm1} was known to Weingarten and Darboux as an example of a metric whose isometric embedding into Euclidean space is integrable by the method of Darboux \cite{Darboux}.  They were also aware of other examples of such metrics, but to date we have not been able to explicitly identify the metrics in Theorem \ref{normalf} in their work.
\end{rem}

\begin{rem}[Prolongation of the isometric embedding system]
In general, if a differential system fails to be Darboux integrable, it may happen that some {\it prolongation} of it is Darboux integrable. One can show that the first prolongation of the isometric embedding system $\CI$ is, in fact, the Gauss-Codazzi system for the embedding.  (This system is
discussed extensively in Section 6.4 in \cite{CfB2}, where it also appears as the prolongation
of the system generated by just $\theta_0,\theta_1,\theta_2$.) A straightforward, albeit somewhat tedious, computation shows that imposing the requirement that the Gauss-Codazzi system be Darboux integrable does not enlarge the class of Darboux integrable 2-metrics.
\end{rem}

\subsection{Embeddings with extrinsic symmetry}
Since each of the above metrics have intrinsic symmetry, we are led to ask whether they can be isometrically
immersed with extrinsic symmetry as well.  In fact, it is an exercise in elementary differential geometry to show that each of them can
be embedded as a surface of revolution in a 1-parameter family of ways, modulo rigid motion.  (In other words, we do not count
as distinct a pair of embeddings that differ only by a rigid motion.)  The single parameter controlling the shape
is $\alpha = |\vV(\theta)|$, the `speed' of the cylindrical coordinate $\theta$ with respect to the canonical Killing
field $\vV = J\nabla q$ along the surface.  For example, in Case 1 the metric is complete, but its embedding
as a surface of revolution is complete only for $\alpha = 3$, in which case one obtains the surface of
revolution $z=\tfrac12(x^2+y^2)$; for $\alpha >3$ the surface comes to a sharp point
where it intersects the axis, while for $\alpha <3$ there is a 1-dimensional boundary circle where $u=0$
(see Figure \ref{parbs}).

\begin{figure}[h]
\begin{center}
\includegraphics[width=1.4in]{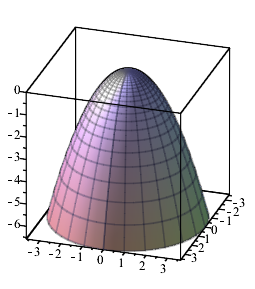}
\includegraphics[width=1.4in]{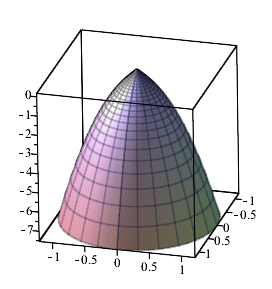}
\includegraphics[width=1.4in]{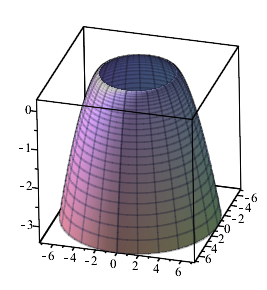}
\includegraphics[width=1.4in]{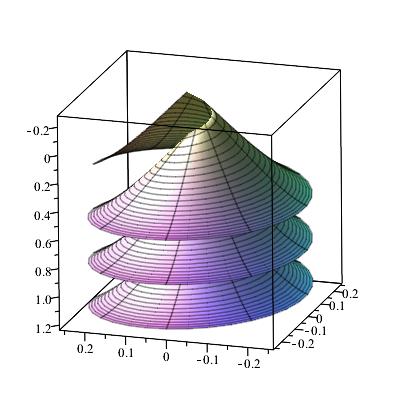}
\end{center}
\caption{Three embeddings of the Case 1 metric as a surface of revolution, and one embedding of the Case 4 metric
as a surface with screw-motion symmetry.}
\label{parbs}
\end{figure}

Of course, vector fields generating rotations are not the most general Killing fields in $\R^3$;
for example, the rightmost surface in Figure \ref{parbs} is the image of an embedding of
the Case 4 metric, where the Killing field on the surface coincides with an ambient Killing field
which generates a screw motion.  A general Killing field
at a point $p \in \R^3$ is given by
\begin{equation}\label{general-Killing}
\vK(p) = \vW + \brho \times \vZ
\end{equation}
where $\vW$ and $\vZ$ are fixed vectors with the latter nonzero, and $\brho$ is vector-valued function giving the position vector of $p$.  If one decomposes $\vW = a \vZ - p_0 \times \vZ$, then
then $\vK$ generates a `screw motion' combining
translation along $\vZ$ with rotation around an axis through $p_0$ parallel to $\vZ$.
We now show how to produce isometric immersions of our metrics such that $\vV$ is the restriction of a general Killing field.

\begin{thm}  Each metric whose isometric embedding system is Darboux-integrable has (modulo rigid motion) a 2-parameter family of local isometric immersions with extrinsic symmetry.
\end{thm}
\begin{proof}  We will determine the equations that a vector field tangent to a surface in $\R^3$ must satisfy in order to be the
restriction of a Killing field.  We will then show that, when we add to these equations the
isometric embedding system for a surface $M$ with a metric $\mg$ satisfying the conditions of Theorem \ref{existence-of-Killing},
and impose the requirement that the embedding maps the Killing field on $M$ to the Killing field on the surface,
we obtain a Frobenius system.

Suppose that $S$ is a surface in $\R^3$ carrying an orthonormal frame field $(\ve_1, \ve_2, \ve_3)$ adapted with $\ve_3$ normal to the surface.
A tangent vector field $\vY$ is the restriction of a Killing field if and
only if there is a constant vector field $\vZ$ such that $\vY(p) - \brho \times \vZ$ is constant.  (Here,
$\brho$ denotes the position vector as a function on $S$.)  These conditions are equivalent to
$d\vZ = 0$ and $d\vY = d\brho \times \vZ$.
To express these conditions in terms of the
moving frame, we let
$$\vY = y_1 \ve_1 + y_2 \ve_2, \qquad \vZ\vert_S = z_1 \ve_1 + z_2 \ve_2 + z_3 \ve_3,$$
for some functions $y_i$ and $z_a$ on $S$.  (In what follows, let
$1\le i,j \le 2$ and $1 \le a,b \le 3$.)   By substituting the expansion of $\vZ$ into the condition $d\vZ=0$, and using \eqref{d-of-frame}, we obtain
\begin{equation}\label{zgroup} d z_a = -z_b \w^a_b. \end{equation}
(Here, $\w^a$, $\w^a_b$ denote the pullbacks of the canonical and connection forms from $\FR$ via the frame field.)
By substituting the expansions of $\vY$ and $\vZ$ into the condition $d\vY = d\brho \times \vZ$, computing
derivatives using \eqref{d-of-basepoint} and \eqref{d-of-frame}, and expanding both sides in terms of our frame field, we obtain the equations
\begin{align}
d y_1 + y_2 \w^1_2 &= z_3 \w^2 \label{x1deriv}\\
d y_2 + y_1 \w^2_1 &= -z_3 \w^1 \label{x2deriv} ,\\
y_1 \w^3_1+y_2 \w^3_2 &= z_2 \w^1 - z_1\w^2. \label{xrel}
\end{align}
Conversely, if $S$ is a surface in $\R^3$ carrying some adapted frame field, and there is a non-trivial solution to these equations,
then $S$ has an extrinsic symmetry, i.e., there is a Killing field which is tangent to $S$.  More formally, we define the {\em Killing system}
for surfaces in $\R^3$ as the following Pfaffian system
$$\calJ = \{ d y_1 + y_2 \w^1_2 - z_3 \w^2, d y_2 + y_1 \w^2_1 + z_3 \w^1, y_1 \w^3_1+y_2 \w^3_2 - z_2 \w^1 + z_1\w^2, dz_a + z_b \w^a_b, \w^3\}.$$
on $\FR \times \R^2 \times \R^3$, where $y_i$ and $z_a$ are coordinates on the last two factors.  Then any integral surface of this
EDS (satisfying the usual independence condition) corresponds to a surface in Euclidean space with a tangential Killing field.

By Theorem \ref{existence-of-Killing}, the metric on $M$ will have local coordinates $s,t$ such that
$\eta^1 = ds$, $\eta^2 = q' dt$ is an orthonormal coframe field with connection form $\eta^1_2 = -q''(s)dt$, where $q(s)$ satisfies the ODE \eqref{qode} and the Killing field is given by $\vV =\di/\di t$.  By fixing this choice of coframing on $M$, we
obtain the isometric embedding system
$$\I_0 = \{ \w^3, \w^1 - ds, \w^2 - q' dt, \w^1_2 + q'' dt\},$$
defined on $M \times \FR$.  (The generators here are the pullbacks of the generators of $\I$ by the section of $\FM$ represented by the
chosen coframing.)  The condition that the embedding take the Killing field $\vV$ on $M$ to the vector field $\vX$ is equivalent
to requiring that $y_1 = 0$ and $y_2 = \eta^2(\vV) = q'$.  When we pull $\calJ$ back to the submanifold where these equations hold, and combine with the
EDS $\I_0$, we obtain
\begin{multline*}
\calK = \I_0 + \calJ = \{ \w^3, \w^1 - ds, \w^2 - q'dt, \w^1_2 + q'' dt, q' (\w^1_2 - z_3 dt),\\
 (q''+z_3)ds, q' (\w^3_2 + z_1 dt) - z_2 ds, dz_a + z_b \w^a_b \}.\end{multline*}
Solutions satisfying the independence condition $ds \wedge dt \ne 0$ exist only on the subset where $z_3 =-q''$, so we pull the EDS back to that
subset, obtaining
\begin{multline*}
\calK_0 = \{ \w^3, \w^1 - ds, \w^2 - q'dt, \w^1_2 + q'' dt,  \w^3_2 + z_1 dt - (z_2/q') ds, dz_1 + q''(w^3_1-z_2 dt), \\
dz_2 + z_2 q''/q' ds, dq'' + (z_1/q'')dz_1 - ((z_2)^2/q')ds \}.
\end{multline*}
We regard this rank 8 Pfaffian system as being defined on
the product $M \times \FR \times \R^2$ (where $z_1, z_2$ are coordinates on the last factor).   It is easy to check that it satisfies the Frobenius condition, so there is an 8-parameter family of solutions (i.e., a unique integral surface through every point in the 10-dimensional product).  Because each surface lies in a 6-dimensional family related by rigid motion, each metric has a 2-parameter family of non-congruent isometric immersions with extrinsic symmetry.
\end{proof}

\begin{rem}
The two parameters occur as first integrals of the above system.  For, along any solution there are constants $\alpha$ and $\beta$ such that
$z_2 = \beta/q'$ and
$$(z_1)^2 + \beta^2/(q')^2 + (q'')^2 = |\vZ|^2 = \alpha^2.$$
Since $\vY \cdot \vZ = q' z_2 = \beta$, the immersions where the Killing field $\vY$ is purely rotational (and the image is a surface
of revolution) are those where $\beta=0$.  In general, the parameter $\alpha$ is the angular velocity of the Killing field, while $\beta$ is proportional to its slope.

Once these two constants are chosen, the embedding is completely determined up to rigid motions.
On the other hand, we know that for each of these metrics there is a family
of embeddings depending on two {\em functions} of one variable.  So, for
a generic embedding $\vV$ will not coincide with the restriction of any extrinsic Killing field.
\end{rem}

\section{Immersions of Lorentzian Surfaces into $\R^{1,2}$}\label{LorentzianSurfaces}
In this section, we study the Lorentzian analog of the problem of \S \ref{ImmersionsToE3} by considering the Darboux integrability of embedding Lorentzian 2-metrics into $1+2$-Minkowski space. Throughout this section, we use the same sign conventions as in \cite{ClellandBook}.

\subsection{The Isometric Embedding System and Integrability Conditions}
Let $\R^{1,2}$ be Minkowski space with an inner product $\langle, \rangle$ of signature $(1,2)$.  Define a frame $(\ve_1, \ve_2, \ve_3)$ as
being {\em orthonormal} if
\[\langle\ve_1,\ve_1\rangle=1,\quad \langle\ve_2,\ve_2\rangle=-1,\quad\langle\ve_3,\ve_3\rangle=-1,\quad \langle\ve_a,\ve_b\rangle=0,\quad\text{for}~a\neq b.\]
Let $\FRm$ be the orthonormal frame bundle of Minkowski space, which carries canonical forms $\w^i$ and connection forms $\w^i_j$
satisfying the usual structure equations \eqref{R3streq}, and basepoint and frame vector functions satisfy \eqref{d-of-basepoint} and
\eqref{d-of-frame}.  However, because the structure group of $\FRm$ is $SO(1,2)$, the connection forms satisfy $\w^i_j=0$ when $i=j$ and
\begin{equation}\label{MinkowskiConnectionForms}
\omega^1_2=\omega^2_1,\quad \omega^1_3=\omega^3_1,\quad \omega^3_2=-\omega^2_3.
\end{equation}

Let $M$ be a Lorentzian surface. Let $\FM$ be the orthonormal frame bundle of $M,$ which carries canonical dual forms $\eta^1,\eta^2$ and connection form $\eta^1_2,$ satisfying the structure equations
\begin{equation}\label{SurfaceStr}
d\eta^1=-\eta^1_2\wedge\eta^2,\quad d\eta^2=-\eta^2_1\wedge\eta^1,\quad d\eta^1_2=-K\eta^1\wedge\eta^2,
\end{equation}
where $\eta^1_2=\eta^2_1$ and $K$ is the Gauss curvature of $M.$

As in the Riemannian case, we define a Pfaffian system $\mathcal{I}$ on $\FM \times\FRm$, generated by the 1-forms
\begin{equation}\label{IsometricEDS}
\theta_0:=\omega^3,\quad \theta_1:=\omega^1-\eta^1,\quad\theta_2:=\omega^2-\eta^2,\quad\theta_3:=\omega^1_2-\eta^1_2 .
\end{equation}
Similar computations to those in the Riemannian case show that $\CI$ is generated algebraically by $\th_0, \th_1, \th_2, \th_3$ together with the 2-forms
\begin{equation}\label{IsometricEDS2forms}
\Omega_0:=\omega^3_1\wedge \eta^1+\omega^3_2\wedge\eta^2,\quad \Omega_1:=\omega^3_1\wedge \omega^3_2 - K\eta^1\wedge\eta^2.
\end{equation}
Decomposable linear combinations $a\Omega_1 + b\Omega_0$ are given by taking $b/a=\pm \sqrt{-K}.$ Thus, $\mathcal{I}$ is hyperbolic when $K<0$ and the elliptic when $K>0.$  (Our introduction of a minus sign in the definition of $K$ in \eqref{SurfaceStr}
allows us to label the Lorentzian cases in this way, so that they are analogous to the Riemannian cases.)

We have the following theorem, whose statement is exactly analogous to those of Theorems \ref{IcondsRiemHyp} and \ref{IcondsRiemEll}, with some
differences in signs and signature.  As in the Riemannian case, in order to treat both the elliptic and hyperbolic cases at the same time, we let
$\epsilon = \operatorname{sgn}(K)$.

\begin{thm}\label{IcondsMink}
Let $\mg$ be a Lorentzian metric on $M$, with Gauss curvature $K$ that is either strictly positive ($\epsilon=1$) or strictly negative ($\epsilon=-1$).
Let $q = |K|^{-3/4}$.  Then $\CI$ is Darboux integrable if and only if the covariant derivatives of $q$ (as defined by \eqref{q-covariant-derivatives}) satisfy the conditions
\begin{equation}\label{qcond-Lorentzian}
q_{11}= -q_{22} = 3 \eps q^{-1/3}, \qquad q_{12}=q_{21}=0.
\end{equation}
\end{thm}

%
%

\subsection{Killing Fields}

As in the Riemannian case, it is straightforward to show that when the function $q$ on $\FM$ satisfies the differential equations \eqref{qcond-Lorentzian}, the vector field
\[ \vV = q_1 \ve_2 - q_2 \ve_1 \]
is a Killing field for the metric $\bsy{g}$.  But in the Lorentzian case, we need to consider separately the possibilities that this vector field is spacelike or timelike (i.e., $\langle \vV, \vV \rangle <0$ or $\langle \vV, \vV \rangle >0$, respectively).

\begin{rem}
In principle, we should also consider the possibility that $\vV$ is lightlike.  This is the case if and only if $q_2 = \pm q_1$, but if this condition holds on any open subset of $\FM$, then equations \eqref{qcond-Lorentzian} are inconsistent.  Therefore, $\vV$ cannot be lightlike on any open subset of $M$, and we will restrict to the open subset of $M$ where $\vV$ is either spacelike or timelike.
\end{rem}

If $\vV$ is spacelike, then we can choose an orthonormal framing $\ve_1, \ve_2$ on $M$ such that $\vV$ is a positive multiple of the spacelike vector $\ve_2$.  On the other hand, if $\vV$ is timelike, then we can choose an orthonormal framing such that $\vV$ is a positive multiple of the timelike vector $\ve_1$.  Having done so, arguments analogous to those given in the Riemannian case may be used to prove the following.

\begin{thm}Suppose that $\bsy{g}$ is a Lorentzian metric on $M$ for which the embedding EDS $\CI$ is Darboux integrable. Then $\bsy{g}$  admits a Killing field $\vV$, and:
\begin{itemize}
\item If $\vV$ is spacelike, then
there exist local coordinates $(s,t)$ on $M$ with respect to which the metric can be written in the form
$$
\bsy{g}=ds^2 - q'(s)^2dt^2,
$$
where $q$ is a strictly monotone solution of the ODE
$$
q''=3\eps q^{-1/3}.
$$
\item If $\vV$ is timelike, then
there exist local coordinates $(s,t)$ on $M$ with respect to which the metric can be written in the form
$$
\bsy{g}=q'(s)^2dt^2 - ds^2,
$$
where $q$ is a strictly monotone solution of the ODE
$$
q''=-3\eps q^{-1/3}.
$$
\end{itemize}
In either case, in these coordinates the Killing field is given by $\displaystyle \vV=\frac{\partial}{\partial t}$.
\end{thm}

\subsection{Normal Forms for Darboux Integrable Metrics}

Normal forms may be obtained by the same integration procedure as that used in \S \ref{normal-forms-Riemannian-sec}.  In the Lorentzian case, the result depends both on the type of the system (hyperbolic or elliptic) and on the type of the Killing field (spacelike or timelike).  Carrying out this procedure yields the following theorems.

\begin{thm}
Let $\bsy{g}$ be an elliptic Darboux integrable Lorentzian metric with a spacelike Killing field on $M$.  Then up to local coordinate transformations on $M$, $\bsy{g}$ is locally equivalent to one of the following:
\begin{enumerate}
\item $\bsy{g} = \cosh^4 u\, (du)^2 - \sinh^2 u\, (dv)^2$, with $K = \sech^4 u$;
\item $\bsy{g} = \sinh^4 u\, (du)^2 - \cosh^2 u\, (dv)^2$, with $K = \csch^4 u$;
\item $\bsy{g} = u^2 \left( (du)^2 - (dv)^2 \right)$, with $K = u^{-4}$.
\end{enumerate}
\end{thm}

\begin{thm}
Let $\bsy{g}$ be an elliptic Darboux integrable Lorentzian metric with a timelike Killing field on $M$.  Then up to local coordinate transformations on $M$, $\bsy{g}$ is locally equivalent to
\[ \bsy{g} = \sin^2 u\, (dv)^2 - \cos^4 u\, (du)^2,\ \text{with} \ K = \sec^4 u.  \]
\end{thm}

\begin{thm}
Let $\bsy{g}$ be a hyperbolic Darboux integrable Lorentzian metric with a spacelike Killing field on $M$.  Then up to local coordinate transformations on $M$, $\bsy{g}$ is locally equivalent to
\[ \bsy{g} = \cos^4 u\, (du)^2 - \sin^2 u\, (dv)^2,\ \text{with} \ K = -\sec^4 u.  \]
\end{thm}

\begin{thm}
Let $\bsy{g}$ be a hyperbolic Darboux integrable Lorentzian metric with a timelike Killing field on $M$.  Then up to local coordinate transformations on $M$, $\bsy{g}$ is locally equivalent to one of the following:
\begin{enumerate}
\item $\bsy{g} = \sinh^2 u\, (dv)^2 - \cosh^4 u\, (du)^2$, with $K = -\sech^4 u$;
\item $\bsy{g} = \cosh^2 u\, (dv)^2 - \sinh^4 u\, (du)^2$, with $K = -\csch^4 u$;
\item $\bsy{g} = u^2 \left( (dv)^2 - (du)^2 \right)$, with $K = -u^{-4}$.
\end{enumerate}
\end{thm}

\begin{rem}
One can also consider the problem of isometrically embedding a Riemannian surface with metric $\mg$ as a spacelike surface in Minkowski
space where the inner product has signature $(2,1)$.
In that case, the relevant exterior differential system is hyperbolic when the
Gauss curvature $K$ of $\mg$ is strictly positive, and elliptic when $K<0$.  However, the set of Riemannian metrics for which this system
is Darboux integrable is exactly the same as in Theorem \ref{normalf}.
\end{rem}

\section{An Explicit Lorentzian Embedding}\label{Example-sec}

In \S \ref{ImmersionsToE3} and \S \ref{LorentzianSurfaces}, we classified those Riemannian and Lorentzian 2-metrics whose isometric embedding EDS
is Darboux integrable on the relevant product of frame bundles. In this section, we study one of these metrics in detail,
namely
\begin{equation}\label{exampleMetric}
\bsy{g}_0=u^2\big(dv^2 - du^2\big).
\end{equation}
For this metric, we show how to derive explicit formulas for its isometric embeddings into $\R^{1,2}$ by making use of its Darboux integrability. We also examine the corresponding geometric Cauchy problem, which asks for an isometrically embedded surface containing a prescribed curve and normal to a prescribed vector field along the curve. We present this as an example of the role played by Darboux integrability in the isometric embedding problem and which applies equally to any of the classified metrics whose embedding system is hyperbolic.

\subsection{The embedding EDS in local coordinates}

In this subsection, we derive a local coordinate expression for the differential system corresponding to the isometric embedding of the Lorentz signature metric
$
\bsy{g}_0
$
into $\R^{1,2}$ with its standard metric $\bsy{h} = dx_1^2 - dx_2^2 - dx_3^2$.
By choosing the specific coframing $\eta^1 = u\, dv$, $\eta^2 = u\, du$ for the metric $\bsy{g}_0$ on $M$, we can pull back the isometric embedding EDS from $\FM \times \FRm$ to $M \times \FRm$.  The coframing $(\eta^1, \eta^2)$ on $M$ satisfies the structure equations \eqref{SurfaceStr}, with $\eta^1_2 = u^{-1}\, dv$ and $K = -u^{-4}$.
Meanwhile, the frame bundle $\FRm$ has coframing $(\o^i, \o^i_j)$, $1\leq i, j\leq 3$, satisfying structure equations \eqref{R3streq}  with the connection forms $\o^i_j$ satisfying the symmetries \eqref{MinkowskiConnectionForms}.

As in \S \ref{LorentzianSurfaces}, the isometric embedding EDS is generated by the 1-forms \eqref{IsometricEDS} and the 2-forms \eqref{IsometricEDS2forms}, and the singular systems are given by
\begin{equation}\label{characteristicSystems}
\pV_\pm=\left\{\th_0,\ \th_1,\ \th_2, \ \th_3, \o^3_1\mp k\eta^2,\ \o^3_2\pm k\eta^1\right\},
\end{equation}
where, because $\bsy{g}_0$ has $K = -u^{-4}$, we have $k=u^{-2}$.

In order to carry out the method of \cite{AFV}, we must construct local coordinates on the ambient manifold  $M\times \FRm$ and express the EDS $\CI$ in terms of these coordinates.  Since $\FRm \simeq \R^{1,2} \times SO(1,2)$ and we have local coordinates $(u,v)$ on $M$ and $(x_1, x_2, x_3)$ on $\R^{1,2}$,  we simply need to construct local coordinates on the $SO(1,2)$ factor.  A rational parametrization may be constructed as follows. Define a basis
\[ \mathbf{b}_1 = \begin{bmatrix} 0 & -1 & 0 \\ -1 & 0 & 0 \\ 0 & 0 & 0 \end{bmatrix}, \qquad
\mathbf{b}_2 = \begin{bmatrix} 0 & 0 & -1 \\0 & 0 & 1 \\ -1 & -1 & 0 \end{bmatrix}, \qquad
\mathbf{b}_3 = \begin{bmatrix} 0 & 0 & -1 \\0 & 0 & -1 \\ -1 & 1 & 0 \end{bmatrix} \]
for the Lie algebra $\mathfrak{so}(1,2)$, and consider the matrix $g$ defined by
\begin{multline}\label{SO12-in-coords}
g(a_1, a_2, a_3)  = \exp(a_3 \mathbf{b}_3)\exp(a_2 \mathbf{b}_2) \exp(\ln(a_1) \mathbf{b}_1) \\[0.1in]
 =
\begin{bmatrix}
\displaystyle{\frac{(a_2 a_3 + 1)^2 + a_2^2 + a_1^2(a_3^2+1)}{2 a_1} } &
\displaystyle{\frac{(a_2 a_3 + 1)^2 + a_2^2 - a_1^2(a_3^2+1)}{2 a_1} } &
-a_2(a_3^2+1) - a_3 \\[0.1in]
\displaystyle{\frac{(a_2 a_3 + 1)^2 - a_2^2 + a_1^2(a_3^2-1)}{2 a_1} } &
\displaystyle{\frac{(a_2 a_3 + 1)^2 - a_2^2 - a_1^2(a_3^2-1)}{2 a_1} } &
-a_2(a_3^2-1) - a_3 \\[0.1in]
\displaystyle{\frac{-a_3(a_2^2 + a_1^2) - a_2}{a_1}} &
\displaystyle{\frac{-a_3(a_2^2 - a_1^2) - a_2}{a_1}} &
2 a_2 a_3 + 1
\end{bmatrix},
\end{multline}
where $a_1 >0$.
If $\bsy{x}=(x_1, x_2, x_3)^t\in\B R^{1,2}$, then the transformation $\bsy{x}\mapsto g\,\bsy{x}$ defined by matrix multiplication preserves the quadratic form $Q(\bsy{x})=x_1^2-x_2^2-x_3^2$.
Letting $H^3$ denote the halfspace $\{a\in\B R^3~|~a_1>0\}$, \eqref{SO12-in-coords} defines a map $g: H^3\hookrightarrow SO(1,2)^+$, where $SO(1,2)^+$ denotes the identity component of the Lorentz group
$SO(1,2)$.  (This consists of matrices $g \in O(1,2)$ that are proper $(\text{det}\,g=1)$ and orthochronous,
i.e., they preserve light-cone orientation.) The image of this mapping is an open dense subset of $SO(1,2)^+$.

Then we identify the frame bundle $\FRm$ with the matrix Lie group consisting of all matrices $G \in GL(4,\R)$ of the form
\begin{equation}
G= \begin{bmatrix}1 & 0 & 0 & 0\cr
                    x_1 & &\cr
										x_2 &&g\cr
										x_3&&\end{bmatrix}
\end{equation}
with $g$ as in \eqref{SO12-in-coords}.  The mapping $(x_i, a_j) \mapsto G$ then covers an open dense subset $\mF_0\subset \FRm$ on which
we will use the $x_i$ and $a_j$ as coordinates.  In terms of these coordinates, the canonical and connection 1-forms of $\FRm$ are
given by the components of the left-invariant Maurer-Cartan form $G^{-1}dG$:
\begin{equation}\label{MCformsE21}
\begin{aligned}
&\o^1=\displaystyle{\frac{(a_2 a_3 + 1)^2 + a_2^2 + a_1^2(a_3^2+1)}{2 a_1} }\, dx_1 - \displaystyle{\frac{(a_2 a_3 + 1)^2 - a_2^2 + a_1^2(a_3^2-1)}{2 a_1} }\, dx_2 \\
& \qquad + \displaystyle{\frac{a_3(a_2^2 + a_1^2) + a_2}{a_1}}\, dx_3, \\[0.1in]
& \o^2= - \displaystyle{\frac{(a_2 a_3 + 1)^2 + a_2^2 - a_1^2(a_3^2+1)}{2 a_1} }\, dx_1 + \displaystyle{\frac{(a_2 a_3 + 1)^2 - a_2^2 - a_1^2(a_3^2-1)}{2 a_1} } \, dx_2 \\
& \qquad + \displaystyle{\frac{-a_3(a_2^2 - a_1^2) - a_2}{a_1}}\, dx_3, \\[0.1in]
& \o^3= (a_2(a_3^2+1) + a_3)\, dx_1 - (a_2(a_3^2-1) + a_3)\, dx_2 + (2 a_2 a_3 + 1)\, dx_3, \\[0.1in]
&\o^1_2= -\displaystyle{\frac{da_1}{a_1} } + 2 a_2\, da_3, \\[0.1in]
& \o^3_1= -\displaystyle{\frac{da_2 + (a_1^2 - a_2^2)\,da_3}{a_1} }, \\[0.1in]
& \o^3_2= -\displaystyle{\frac{da_2 - (a_1^2 + a_2^2)\, da_3}{a_1} }.
\end{aligned}
\end{equation}

\subsection{Adapted frames for the Darboux pair}\label{AdaptedFrames}
We now focus on using the Darboux integrability of the metric (\ref{exampleMetric}) to construct representations of all local isometric embeddings
$\iota: (M,\bsy{g}_0)\hookrightarrow (\B R^{1,2}, \bsy{h})$. We do this by first constructing the {\it 5-adapted frames} and {\it Vessiot group} for the {\it Darboux pair} whose respective annihilators we denote by $\CH_\pm$. These objects,  described in detail in \cite{AFV}, are constructed below starting from the EDS $\CI$.The integration procedure established in \cite{AFV} is then used to construct the integral submanifolds of $\CI$.

We have $k=u^{-2}$ and $dk=k_1\eta^1+k_2\eta^2$, which implies that $k_1=0,\ k_2=-2u^{-4}$. By design, the singular systems (\ref{characteristicSystems}) have derived flags which terminate in rank 2 integrable sub-bundles given by
\begin{equation}\label{g0chars}
\begin{aligned}
\pV^{\infty}_+=&\left\{\o^3_1-k\eta^2+\frac{k_1}{2k^2}(\eta^1_2-\o^1_2),\ \o^3_2+k\eta^1+\frac{k_2}{2k^2}(\eta^1_2-\o^1_2)\right\} = \left\{\o^3_1 - \frac{du}{u}, \w^3_2 + \w^1_2 \right\}, \cr
\pV^{\infty}_-=&\left\{\o^3_1+k\eta^2-\frac{k_1}{2k^2}(\eta^1_2-\o^1_2),\ \o^3_2-k\eta^1-\frac{k_2}{2k^2}(\eta^1_2-\o^1_2)\right\} = \left\{\o^3_1 + \frac{du}{u}, \w^3_2 - \w^1_2 \right\}.
\end{aligned}
\end{equation}
We denote the space of first integrals of $\pV_\pm$ by $\inv\pV_\pm$ and find that
$$
\inv\pV_+=\left\{ \frac{a_3(a_1+a_2)+1}{a_1+a_2},\ \frac{\sqrt{u}(a_1+a_2)}{\sqrt{a_1}}\right\}, \
\inv\pV_-=\left\{ \frac{a_3(a_1-a_2)-1}{a_1-a_2},\ \frac{\sqrt{u}(a_1-a_2)}{\sqrt{a_1}}\right\},
$$
where we have assumed that $u>0$.  That is, any first integral of $\pV_\pm$ is a function of the elements of
$\inv\pV_\pm$.
For the sake of convenience we set
\begin{equation}\label{characteristicFirstIntegrals}
\begin{aligned}
&p= \frac{a_3(a_1-a_2)-1}{a_1-a_2}, \qquad p_0=\frac{\sqrt{u}(a_1-a_2)}{\sqrt{a_1}},\cr
&q= \frac{a_3(a_1+a_2)+1}{a_1+a_2}, \qquad q_0= \frac{\sqrt{u}(a_1+a_2)}{\sqrt{a_1}}.
\end{aligned}
\end{equation}

\begin{rem} One way to arrive at these integrals is to carefully examine the structure equations of the singular systems
given at the far right in \eqref{g0chars}.  Taking the `plus' system for example, one finds that
$$d(\w^1_2 + \w^3_2) = (\w^1_2 + \w^3_2) \wedge \w^3_1,$$
indicating that $\w^1_2 + \w^3_2$ is an integrable 1-form---i.e., it is locally, up to a nonzero multiple, the exact derivative of a function.
This function arises as follows. From \eqref{d-of-frame} we have
$$d(\ve_1 - \ve_3) = -(\ve_1 - \ve_3) \w^3_1 + \ve_2 (\w^1_2 + \w^3_2).$$
Recall that the function $\vn=\ve_1-\ve_3$ takes value in the cone $\CN \subset \R^{1,2}$ of nonzero null vectors.  The above equation
indicates that the projectivization of $\ve_1 - \ve_3$ is a first integral of the 1-form $\w^1_2 + \w^3_2$.  In other words,
if $\proj:\CN \to \RP^1$ is the projectivization map, then $d(\proj \circ \vn) \equiv 0$ modulo $\w^1_2 + \w^3_2$.
Hence, the pullback of any local coordinate on the projectivized null cone $\RP^1$ is a first integral of $\pV_+$.  Moreover, the above equation implies that
$$d(\ve_1 - \ve_3) \equiv  -(\ve_1 - \ve_3) \frac{du}{u} \quad \mod \ \pV^{\infty}_+,$$
and it follows that the null vector $u(\ve_1 -\ve_3)$ is a first integral of $\pV^{\infty}_+$.  In fact, subtracting the first and third
columns of the right-hand side of \eqref{SO12-in-pq-coords} and multiplying by $u$ gives
$$u(\ve_1 - \ve_3) = \left( \tfrac12 q_0^2 (q^2 + 1), \tfrac12 q_0^2 (q^2 -1), -q_0^2 q \right)^t.$$
(Here we have used the relation $u = -\tfrac{1}{2}p_0 q_0 (p-q)$, which is a straightforward consequence of \eqref{characteristicFirstIntegrals}.)
Because $\CN$ is two-dimensional, we obtain two independent first integrals this way.  The first integrals $p,p_0$ of $\pV^{\infty}_-$ arise in a similar fashion by computing $u(\ve_1 + \ve_3)$.
\end{rem}

In accordance with the procedure set down in \cite{AFV}, we pass to a coordinate system adapted to $\inv\pV_\pm$.
Let $N \subset M\times \mF_0$ be the open subset where $a_1 - a_2 \ne 0$ and $a_1 + a_2 \ne 0$,
and let
$$\phi:N \to \R^8$$
denote the  mapping defined by the coordinate transformation
$$\phi (u, v, a_j, x_i) = (p, p_0, q, q_0, v, x_i).$$
This is a diffeomorphism onto its image $N_1 \subset\R^8$.

With respect to this coordinate system, the local parametrization \eqref{SO12-in-coords} for the $SO(1,2)$ component of $\FRm \simeq \R^{1,2} \times SO(1,2)$ may be expressed as
\begin{equation}\label{SO12-in-pq-coords}
g\left(p, q, \frac{p_0}{q_0}\right)  =
\begin{bmatrix}
\displaystyle{-\frac{p_0^2(p^2+1) + q_0^2(q^2+1)}{2 p_0 q_0 (p-q)} } &
\displaystyle{\frac{pq+1}{p-q} } &
\displaystyle{-\frac{p_0^2(p^2+1) - q_0^2(q^2+1)}{2 p_0 q_0 (p-q)} } \\[0.1in]
\displaystyle{-\frac{p_0^2(p^2-1) + q_0^2(q^2-1)}{2 p_0 q_0 (p-q)} } &
\displaystyle{\frac{pq-1}{p-q} } &
\displaystyle{-\frac{p_0^2(p^2-1) - q_0^2(q^2-1)}{2 p_0 q_0 (p-q)} } \\[0.1in]
\displaystyle{\frac{p_0^2 p + q_0^2 q}{p_0 q_0 (p-q)}} &
\displaystyle{-\frac{p+q}{p-q}} &
\displaystyle{\frac{p_0^2 p - q_0^2 q}{p_0 q_0 (p-q)}}
\end{bmatrix}.
\end{equation}
From this expression, we see that the domain $N_1$ of this coordinate system must be contained in the region where $p \neq q$, $p_0 q_0 \neq 0$.

Now consider the characteristic distributions
$$
H_+=\ann \pV_+,\qquad H_-=\ann \pV_-,\qquad H =\ann \CI = H_+ \oplus H_- .
$$
Substituting \eqref{IsometricEDS} into \eqref{characteristicSystems} to get the $\th_a$ in terms of the $\w^j_i$, and then using (\ref{MCformsE21}),
shows that the pullbacks of the singular systems $\pV_\pm$ to $N_1$ via the diffeomorphism $\phi^{-1}$ may be written as
\begin{equation}\label{char-systems-0}
\begin{aligned}
\pV^1_+ = \Big{\{} & dq,\ dq_0,\ dv + \tfrac{1}{2} p_0^2\, dp,  \\
& dx_1 + \tfrac{1}{8} p_0^2(p_0^2(p^2+1) + q_0^2(q^2 - 2pq -1))\, dp - \tfrac{1}{4}p_0 q_0^2 (p-q)(pq+1)\, dp_0, \\
& dx_2 + \tfrac{1}{8} p_0^2(p_0^2(p^2-1) + q_0^2(q^2 - 2pq +1))\, dp - \tfrac{1}{4}p_0 q_0^2 (p-q)(pq-1)\, dp_0, \\
& dx_3 - \tfrac{1}{4}p p_0^2(p_0^2-q_0^2)\, dp + \tfrac{1}{4} p_0 q_0^2(p^2-q^2)\, dp_0 \Big{\}}, \\[0.05in]
\pV^1_- = \Big{\{} & dp,\ dp_0,\ dv - \tfrac{1}{2} q_0^2\, dq,  \\
& dx_1 - \tfrac{1}{8} q_0^2(p_0^2(p^2 - 2pq -1) + q_0^2(q^2+1))\, dq - \tfrac{1}{4}p_0^2 q_0 (p-q)(pq+1)\, dq_0, \\
& dx_2 - \tfrac{1}{8} q_0^2(p_0^2(p^2 - 2pq +1) + q_0^2(q^2-1))\, dq - \tfrac{1}{4}p_0^2 q_0 (p-q)(pq-1)\, dq_0, \\
& dx_3 - \tfrac{1}{4}q q_0^2(p_0^2-q_0^2)\, dq + \tfrac{1}{4} p_0^2 q_0(p^2-q^2)\, dq_0 \Big{\}}.
\end{aligned}
\end{equation}
Moreover, the pullback $\CI^1$ of $\CI$ to $N_1$ is generated by the Pfaffian system $\pV^1 = \pV^1_+ \cap \pV^1_-$ together with the 2-forms $dp \wedge dp_0$ and $dq \wedge dq_0$.

It follows that the push-forwards of the characteristic distributions $\CH_{\pm}$ to $N_1$ via $\phi$ are given by
\[ \CH_+^1=\{X_1,\ X_2\},\qquad \CH_-^1=\{Y_1, \ Y_2\}. \]
where
\begin{equation}\label{char-distributions-0}
\begin{aligned}
X_1 =  & \P {p_0} + \tfrac{1}{4}p_0 q_0^2 (p-q)(pq+1)\, \P {x_1}
+ \tfrac{1}{4}p_0 q_0^2 (p-q)(pq-1)\, \P {x_2} - \tfrac{1}{4} p_0 q_0^2(p^2-q^2)\, \P {x_3},
\\
X_2=& \P p - \tfrac{1}{2} p_0^2\, \P v - \tfrac{1}{8} p_0^2(p_0^2(p^2+1) + q_0^2(q^2 - 2pq -1))\, \P {x_1} \\
& \qquad - \tfrac{1}{8} p_0^2(p_0^2(p^2-1) + q_0^2(q^2 - 2pq +1))\, \P {x_2} + \tfrac{1}{4}p p_0^2(p_0^2-q_0^2)\, \P {x_3},\\[0.05in]
Y_1 =  &\P {q_0} + \tfrac{1}{4}p_0^2 q_0 (p-q)(pq+1)\, \P {x_1} + \tfrac{1}{4}p_0^2 q_0 (p-q)(pq-1)\, \P {x_2} - \tfrac{1}{4} p_0^2 q_0(p^2-q^2)\, \P {x_3}, \\
Y_2=& \P q + \tfrac{1}{2} q_0^2\, \P v + \tfrac{1}{8} q_0^2(p_0^2(p^2 - 2pq -1) + q_0^2(q^2+1))\, \P {x_1} \\
& \qquad + \tfrac{1}{8} q_0^2(p_0^2(p^2 - 2pq +1) + q_0^2(q^2-1))\, \P {x_2} + \tfrac{1}{4}q q_0^2(p_0^2-q_0^2)\, \P {x_3}.
\end{aligned}
\end{equation}
It is well known (and straightforward to check) that
$[\CH_+, \CH_-]\equiv 0\mod\CH.$
Moreover, the basis vectors above satisfy this congruence exactly; i.e.,
we have
$$
[X_i, Y_j]=0, \ \ \text{for all}\ \ i,j\in\{1,2\}.
$$
This implies, for instance, that the distribution $\CB=\{X_1,\, Y_1\}$ on $N_1$ is a rank 2 integrable distribution,
and it is straightforward to check that its first integrals are spanned by the functions\footnote{We can equally choose $\CB=\{X_i, Y_j\}$
for any $i,j\in\{1,2\}$ since they are all Frobenius integrable and each choice has the desired property of giving bases for $\CH_\pm$ in which one of the basis elements is locally expressible as a coordinate vector field. This provides a means by which to ultimately achieve 5-adapted coframes.}
$$\operatorname{inv}\CB=\{ p,q,v,y_1,y_2,y_3\},$$
where
\begin{equation}\label{invariants-6D-B}
\begin{aligned}
{y}_1 &= x_1 - \tfrac{1}{8} p_0^2 q_0^2 (p-q)(1+pq),\\
{y}_2 &= x_2 + \tfrac{1}{8} p_0^2 q_0^2 (p-q)(1-pq),\\
{y}_3 &= x_3 + \tfrac{1}{8} p_0^2 q_0^2 (p^2 - q^2).
\end{aligned}
\end{equation}
Continuing to follow \cite{AFV}, we achieve the final adapted frame by making a local change of variables
$
\psi:N_1\to N_2
$,
defined by
\begin{equation}\label{define-psi}
\psi(p, p_0, q, q_0, v, x_i) = (p, p_0, q, q_0, v, {y}_i),
\end{equation}
where ${y}_1, {y}_2, {y}_3$ are as in \eqref{invariants-6D-B}.
Straightforward calculations show that the pullbacks of the singular systems $\pV^1_\pm$ via the diffeomorphism $\psi^{-1}$ may be written as
\begin{equation}\label{char-systems-1}
\begin{aligned}
\pV^2_+ = \Big{\{} & dq,\ dq_0,\ dv + \tfrac{1}{2} p_0^2\, dp, \
dy_1 + \tfrac{1}{8} p_0^4(p^2 + 1)\, dp, \
dy_2 + \tfrac{1}{8} p_0^4(p^2 - 1)\, dp, \
dy_3 - \tfrac{1}{4}p p_0^4\, dp \Big{\}}, \\[0.05in]
\pV^2_- = \Big{\{} & dp,\ dp_0,\ dv - \tfrac{1}{2} q_0^2\, dq,  \
dy_1 - \tfrac{1}{8} q_0^4(q^2+1)\, dq, \
dy_2 - \tfrac{1}{8} q_0^4(q^2-1)\, dq , \
dy_3 + \tfrac{1}{4} q q_0^4\, dq \Big{\}},
\end{aligned}
\end{equation}
and hence that the pullback $\CI^2$ of $\CI^1$ by $\psi^{-1}$ is generated by the Pfaffian system
\begin{equation}
\begin{aligned}
\pV^2_+ \cap \pV^2_- = \Big{\{}
& dv + \tfrac{1}{2} p_0^2\, dp - \tfrac{1}{2} q_0^2\, dq, \
dy_1 + \tfrac{1}{8} p_0^4(p^2 + 1)\, dp - \tfrac{1}{8} q_0^4(q^2-1)\, dq, \\
& dy_2 + \tfrac{1}{8} p_0^4(p^2 - 1)\, dp - \tfrac{1}{8} q_0^4(q^2-1)\, dq ,  \
dy_3 - \tfrac{1}{4}p p_0^4\, dp + \tfrac{1}{4} q q_0^4\, dq
\Big{\}},
\end{aligned}
\end{equation}
together with the 2-forms $dp \wedge dp_0$ and $dq \wedge dq_0$.

It follows that the push-forwards of the characteristic distributions $\CH^1_{\pm}$ by $\psi$ are given by
\begin{equation}\label{char-distributions-1}
\begin{aligned}
&\CH^2_+ = \left\{\P {p_0},\, P = \P p - \tfrac{1}{2} p_0^2\, \P v - \tfrac{1}{8} p_0^4(p^2 + 1)\,  \P {y_1} - \tfrac{1}{8} p_0^4(p^2 - 1)\, \P {y_2} + \tfrac{1}{4}p p_0^4\, \P {y_3} \right\}\cr
&\CH^2_- = \left\{\P {q_0},\, Q = \P q + \tfrac{1}{2} q_0^2\, \P v + \tfrac{1}{8} q_0^4(q^2+1)\, \P {y_1} + \tfrac{1}{8} q_0^4(q^2-1)\, \P {y_2} - \tfrac{1}{4} q q_0^4\, \P {y_3} \right\}.
\end{aligned}
\end{equation}
(Note that the vector fields $\P {p_0}, \P{p}, \P {q_0}, \P{q}$ in \eqref{char-distributions-1} are defined relative to the coordinate system $(p, p_0, q, q_0, v, y_i)$ and hence are not the same as those in \eqref{char-distributions-0}.)

In what follows, we will find integral manifolds of $\CI$ by constructing integral manifolds of
$\CI^2$ in $N_2$ and mapping them to $M \times \FRm$ via $\psi^{-1}\circ\phi^{-1}$.
For this purpose, we note that $\pV^2 = \ann \CH^2$, where $\CH^2$ denotes $\CH^2_+ \oplus \CH^2_-$.

\subsection{The Vessiot algebra and superposition} The purpose of the frame adaptations leading to $\CH^2$ is that they enable us to construct the superposition formula from knowledge of the Vessiot algebra associated to any Darboux integrable system, such as the embedding EDS $\CI$. The Vessiot algebra $\mathfrak{vess}(\pV_+,\pV_-)$ of any Darboux pair $(\pV_+,\pV_-)$ permits one to construct a formula
(the superposition formula) which intertwines the integral manifolds of each singular system $\V_\pm$ or, equivalently $\CH_\pm$ to give an integral submanifold of $\CI$; see \S \ref{Darboux-int-sec}.

Indeed, let
$$
\mathfrak{g}_+=\{\P {y_1},\ \P {y_2},\ \P {y_3},\ \P v\}.
$$
Then it is easy to verify that
$$
\bsy{f}^5_+=\mathfrak{g}_+\oplus\{\P {p_0},\ P\}\oplus\{\P {q_0}, \ Q\}
$$
is one of the 5-adapted frames as defined in \cite{AFV} and that $\mathfrak{g}_+$ is the ``left" Vessiot algebra.
Since in this case the Vessiot algebra is abelian, the ``right" Vessiot algebra is equal to the left and the other 5-adapted frame, $\bsy{f}^5_-$, is the same as
$\bsy{f}^5_+$.  In general, the left and right Vessiot algebras of a Darboux pair coincide with the left- and right-invariant vector fields on a Lie group - the Vessiot group of the Darboux pair. The superposition formula corresponds to multiplication on the Vessiot group. In this case, the Vessiot group being abelian, the superposition formula is essentially identical with linear superposition. Thus the procedure for constructing integral submanifolds of $\CH^2$ is, roughly speaking, to separately construct integral submanifolds of $\CH_+^2$ and $\CH_-^2$ and then add the result.  For further explanation and examples, we refer the reader to \cite{ClellandVassiliou} and
\cite{AFV}.

More specifically, in this case the manifolds $\wh{M}_1, \wh{M}_2$ of Theorem \ref{mainAFVthm} are integral manifolds of the systems
\[ (\pV_+^2)^{\infty} = \{ dq, dq_0\}, \qquad (\pV_-^2)^{\infty} = \{ dp, dp_0\}, \]
respectively. They may each be identified with $\R^6$, with local coordinates $(p$, $p_0$, $v^+$,
$y^+_1$, $y^+_2$, $y^+_3)$ on $\wh{M}_1$ and $(q, q_0, v^-, y^-_1, y^-_2, y^-_3)$ on $\wh{M}_2$.  The Pfaffian systems $\wh{\bsy{\theta}}_1, \wh{\bsy{\theta}}_2$ on $\wh{M}_1, \wh{M}_2$ are the pullbacks to these integral manifolds of the singular systems $\pV_+^2$ and $\pV_-^2$, respectively.  The superposition formula combines integral curves $\sigma_+:(a,b) \to \wh{M}_1$ and $\sigma_-:(a,b) \to \wh{M}_2$ of these systems,
$\wh{\bsy{\theta}}_1, \wh{\bsy{\theta}}_2$, to form an integral surface
$\iota_2 = (\sigma_+ * \sigma_-):(a,b) \times (a,b) \to N_2$ of $\CI^2$.  Explicitly, if we write
\begin{align*}
\sigma_+(t) & = (p(t), p_0(t), v^+(t), y^+_1(t), y^+_2(t), y^+_3(t)), \\
\sigma_-(t) & = (q(t), q_0(t), v^-(t), y^-_1(t), y^-_2(t), y^-_3(t)),
\end{align*}
then
\begin{multline}\label{sp-formula}
\iota_2 (t_1, t_2) = \sigma_+(t_1) * \sigma_-(t_2) =
\Big{(} p(t_1), \ p_0(t_1), \ q(t_2), \ q_0(t_2), \\
v^+(t_1) + v^-(t_2),\ \ y^+_1(t_1) + y^-_1(t_2),\ \ y^+_2(t_1) + y^-_2(t_2),\ \ y^+_3(t_1) +
y^-_3(t_2)\Big{)}.
\end{multline}

\subsection{Integral submanifolds of $\CH^2$}\label{construct-int-mflds-sec}

It can be shown that the systems $\wh{\bsy{\theta}}_1, \wh{\bsy{\theta}}_2$ cannot be integrated in finite terms of arbitrary functions and their derivatives alone. On the other hand, it is easy to express the solutions via quadrature. Now we require integral manifolds of $\CI$ to be such that $\eta^1\we\eta^2=du\we dv$ is non-zero. It is easy to see that there are one-dimensional integral manifolds of $\wh{\bsy{\theta}}_1$ upon which $dp\neq 0$. Similarly, there are one-dimensional integral manifolds of $\wh{\bsy{\theta}}_2$ upon which $dq\neq 0$. One can find such integral manifolds
of $\wh{\bsy{\theta}}_1$  by solving the ODE system
\begin{equation}\label{V2plus-ODEs}
\begin{alignedat}{2}
(y^+_1)'(p) & = -\tfrac{1}{8}(p^2+1) f(p)^4, & \qquad
(y^+_2)'(p) & = -\tfrac{1}{8}(p^2-1) f(p)^4, \\
(y^+_3)'(p) & = \tfrac{1}{4}p f(p)^4, & \qquad
(v^+)'(p) & = -\tfrac{1}{2} f(p)^2,
\end{alignedat}
\end{equation}
while integral curves of $\wh{\bsy{\theta}}_2$ are given by solving the ODE system
\begin{equation}\label{V2minus-ODEs}
\begin{alignedat}{2}
(y^-_1)'(q) & = \tfrac{1}{8}(q^2+1) g(q)^4, & \qquad
(y^-_2)'(q) & = \tfrac{1}{8}(q^2-1) g(q)^4, \\
(y^-_3)'(q) & = -\tfrac{1}{4}q g(q)^4, & \qquad
(v^-)'(q) & = \tfrac{1}{2} g(q)^2.
\end{alignedat}
\end{equation}
If we take
\[ f(p) = (8 F'''(p))^{1/4}, \qquad g(q) = (8 G'''(p))^{1/4} \]
for some (arbitrary) smooth functions $F, G$ with $F''', G''' > 0$, then a straightforward integration by parts yields characteristic curves $\sigma_+:(a,b) \to \wh{M}_1$ and $\sigma_-:(a,b) \to \wh{M}_2$ of the form
\begin{equation}
\begin{aligned}
\s_+(p)=\Big{(} & p,\ p_0 = (8 F'''(p))^{1/4}, \ v^+ = -\int \sqrt{2 F'''(p)}\, dp, \\
& y^+_1 = -(p^2+1) F''(p) + 2p F'(p) - 2 F(p), \\
& y^+_2 = -(p^2-1) F''(p) + 2p F'(p) - 2 F(p), \\
& y^+_3 = 2p F''(p) - 2 F'(p) \Big{)}, \\[0.05in]
\s_-(q)=\Big{(} & q,\ q_0 = (8 G'''(q))^{1/4}, \  v^- = \int \sqrt{2 G'''(q)}\, dq, \\
& y^-_1 = (q^2+1) G''(q) - 2q G'(q) + 2 G(q), \\
& y^-_2 = (q^2-1) G''(q) - 2q G'(q) + 2 G(q), \\
& y^-_3 = -2q G''(q) + 2 G'(q) \Big{)}.
\end{aligned}
\end{equation}
Then from the superposition formula, the general integral manifold of $(N_2, \CI^2)$ may be expressed as
\begin{equation}\label{gen-int-mfld-N2}
\begin{aligned}
\iota_2(p,q) & = \s_+(p) * \s_-(q)  = \\
\Big{(} & p,\ p_0 = (8 F'''(p))^{1/4}, q,\ \ q_0 = (8 G'''(q))^{1/4}, \\
& v = -\int \sqrt{2 F'''(p)}\, dp + \int \sqrt{2 G'''(q)}\, dq, \\
& y_1 = -(p^2+1) F''(p) + 2p F'(p) - 2 F(p) + (q^2+1) G''(q) - 2q G'(q) + 2 G(q), \\
& y_2 = -(p^2-1) F''(p) + 2p F'(p) - 2 F(p) + (q^2-1) G''(q) - 2q G'(q) + 2 G(q), \\
& y_3 = 2p F''(p) - 2 F'(p) -2q G''(q) + 2 G'(q) \Big{)} .
\end{aligned}
\end{equation}

Composing the expression \eqref{gen-int-mfld-N2} with the diffeomorphism $\phi^{-1} \circ \psi^{-1}:N_2 \to M \times \FRm$ followed by the projection $\pi:M \times \FRm \to \R^{1,2}$
gives a formula for the general isometric immersion of $(M, \bsy{g}_0)$ into $\R^{1,2}$, but parametrized with respect to the variables $(p,q)$ rather than the original coordinates $(u,v)$ on $M$.  This yields the following parametrization for the general isometric immersion of $(M, \bsy{g}_0)$ into $\R^{1,2}$:
\begin{equation}\label{gen-int-mfld-M}
\begin{aligned}
x_1 & = (p-q)(pq+1)\sqrt{F'''(p) G'''(q)} - (p^2+1)F''(p) + 2p F'(p) - 2 F(p) \\
& \qquad + (q^2+1) G''(q) - 2q G'(q) + 2 G(q), \\
x_2 & = (p-q)(pq-1)\sqrt{F'''(p) G'''(q)} - (p^2-1) F''(p) + 2p F'(p) - 2 F(p) \\
& \qquad + (q^2-1) G''(q) - 2q G'(q) + 2 G(q), \\
x_3 & = (q^2 - p^2)\sqrt{F'''(p) G'''(q)} + 2p F''(p) - 2 F'(p) -2q G''(q) + 2 G'(q).
\end{aligned}
\end{equation}
Of particular significance is the fact that the parametrization \eqref{gen-int-mfld-M} is expressed in terms of two arbitrary functions and their derivatives, {\em without} any integration required.

We also obtain the following expressions for the coordinates $(u,v)$ in terms of $(p,q)$:
\begin{equation}\label{pq-to-uv}
 u = -(p-q) \left(F'''(p) G'''(q)\right)^{1/4}, \qquad
v =  -\int \sqrt{2 F'''(p)}\, dp + \int \sqrt{2 G'''(q)}\, dq .
\end{equation}
For most choices of functions $F,G$ the map $(p,q)\mapsto (u,v)$ has local inverses and hence we have achieved our aim of finding local isometric immersions of $(M,\bsy{g}_0)$ into $\R^{1,2}$ with its standard metric. In most cases, however (depending on the particular functions $F(p), G(q)$), we will not be able to {\it  explicitly} invert in order to obtain an explicit parametrization for the immersion in terms of $(u,v)$.

\begin{xmpl}
For an explicit example, suppose that the functions $f(p)$ and $g(q)$ are constants, say $p_0=\e_1, q_0=\e_2$. This corresponds to choosing $F(p) = \frac{1}{48}\e_1^4 p^3, \ G(q) = \frac{1}{48}\e_2^4 q^3$.  In this case, the map
\[ \iota_0 = \phi^{-1} \circ \psi^{-1} \circ \iota_2: \R^2 \to M \times \FRm \]
defining the corresponding integral manifold of $\CI$ is given by
\begin{align*}
\iota_0(p,q) = \Big{(} & u = \tfrac{1}{2} \e_1 \e_2 (q-p), \ v =  \tfrac{1}{2} (\e_2^2 q - \e_1^2 p), \\
& x_1 = \tfrac{1}{8} \left( -\tfrac{1}{3}\e_1^4 p^3 +\tfrac{1}{3}\e_2^4 q^3 + \e_1^2 \e_2^2 pq(p-q) - (\e_1^2 - \e_2^2)(p+q) \right), \\
& x_2 = \tfrac{1}{8} \left( -\tfrac{1}{3}\e_1^4 p^3 +\tfrac{1}{3}\e_2^4 q^3 + \e_1^2 \e_2^2 pq(p-q) + (\e_1^2 - \e_2^2)(p+q) \right), \\
& x_3 = \tfrac{1}{8} (\e_1^2 - \e_2^2)(\e_1^2 p^2 + \e_2^2 q^2), \\
& a_1 = \frac{(\e_1 + \e_2)^2}{2 \e_1 \e_2 (p-q)}, \
a_2 = -\frac{(\e_1^2 - \e_2^2 )}{2 \e_1 \e_2 (p-q)}, \
a_3 = \frac{(\e_1 p + \e_2 q)}{(\e_1 + \e_2)} \Big{)}.
\end{align*}
In this case, the equations \eqref{pq-to-uv} can be solved for $p$ and $q$ explicitly, and this yields the following parametrization for the immersed surface $\iota = \pi \circ \iota_0$ in terms of the original coordinates $(u,v)$:
\begin{equation}\label{special-int-mfld-M}
\begin{aligned}
x_1 & = \frac{(\e_1^2+\e_2^2)(v^3+3u^2v) - 2\e_1\e_2(u^3+3uv^2)}{3(\e_1^2-\e_2^2)^2} + \frac{1}{4}(\e_1^2+\e_2^2)v - \frac{1}{2}\e_1\e_2u, \\
x_2 & = \frac{(\e_1^2+\e_2^2)(v^3+3u^2v) - 2\e_1\e_2(u^3+3uv^2)}{3(\e_1^2-\e_2^2)^2} - \frac{1}{4}(\e_1^2+\e_2^2)v + \frac{1}{2}\e_1\e_2u, \\
x_3 & = \frac{(\e_1^2+\e_2^2)(u^2+v^2) - 4\e_1\e_2uv}{2(\e_1^2-\e_2^2)}.
\end{aligned}
\end{equation}

As an aide to visualization, introduce a change of variables on $(M,\bsy{g}_0)$, by setting $\bar{u}=v-u,\ \bar{v}=v+u$. Then $\ell^*\bsy{\bar{g}}_0=\bsy{g}_0$, where $\bsy{\bar{g}}_0=\ds \left(\frac{\bar{v}-\bar{u}}{2}\right)^2d\bar{u}\,d\bar{v}$. Below, in Figure \ref{specialImmersions}, we exhibit  graphs of the isometric  immersion of $\bsy{\bar{g}}$ in $(\mathbb{R}^{1,2}, \bsy{h})$ together with some of the coordinates lines for the particular parameter values $\epsilon_1=1, \ \epsilon_2=4$.

\begin{figure}[ht]
\begin{center}
\includegraphics[width=2.5in]{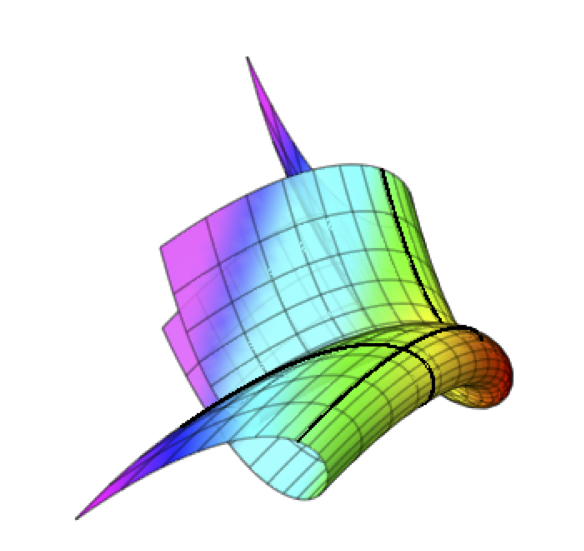}
\includegraphics[width=1.5in]{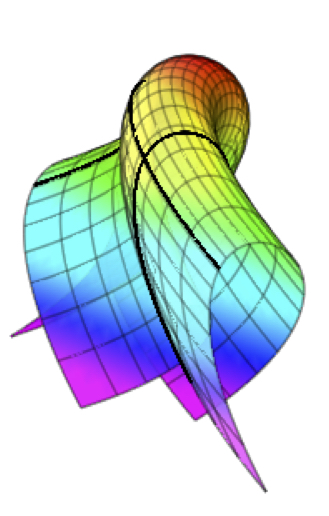}
\includegraphics[width=2.5in]{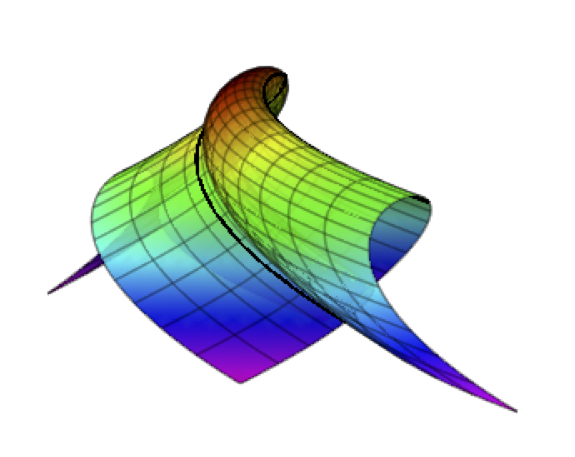}
\end{center}
\caption{\small Views of immersion $\iota$ of $\bsy{\bar{g}}_0$ into $(\mathbb{R}^{1,2}, \bsy{h})$, $\epsilon_1=1, \epsilon_2=4$.}
\label{specialImmersions}
\end{figure}

\end{xmpl}

\subsection{The geometric Cauchy problem for $(M,\bsy{g}_0)\hookrightarrow (\B R^{1,2}, \bsy{h})$}

In this subsection, we briefly explore the role that Darboux integrability plays in resolving the local geometric Cauchy problem for Darboux integrable metrics like $\bsy{g}_0$. Classically, for the local geometric Cauchy problem for $\bsy{g}_0$, one prescribes a smoothly immersed curve $\g: (a,b) \to \B R^{1,2}$ and a (necessarily spacelike) unit vector field $\bar{\ve}_3:(a,b) \to \B R^{1,2}$ along $\g$ orthogonal to the tangent vector field $\g'$.  This initial data determines a unique local isometric immersion of $\bsy{g}_0$ into $\B R^{1,2}$ which contains an open subset of the image of $\g$ and whose normal vector field along this subset is given by $\bar{\ve}_3$; the Cauchy problem seeks to construct this immersion from the given initial data.

We will show in this subsection how the Darboux integrability of $\bsy{g}_0$ leads to a solution of this problem via ODE methods.  Specifically,
for a given initial curve $\gamma$ and normal vector field $\bar{\ve}_3$,
the problem reduces to a system of two first-order ordinary differential equations for two unknown functions---or equivalently, a single second-order scalar ODE.  A solution to this ODE may be thought of as determining a preferred parametrization for the initial data, and after reparametrization, the explicit solution of the geometric Cauchy problem is reducible to quadrature.  For more details regarding these methods and an intrinsic formulation we refer to \cite{AF2}.


Let $\pV^2 = \pV^2_+ \cap \pV^2_-$ (where $\pV^2_\pm$ are as in \eqref{char-systems-1}) denote the Pfaffian system on
$N_2:=\psi \circ \phi(N)$, which corresponds via pullback to the degree one piece of $\CI$, and which is dual to the embedding distribution $\CH^2=\CH^2_+\oplus\CH^2_-$, which has been adapted to the Darboux invariants $\inv \pV_+$ and  $\inv \pV_-$.
As described in \cite{AF2} (and due to the superposition formula \eqref{sp-formula}), the Cauchy problem for this system may be solved as follows: Given a non-characteristic integral curve $\sigma:(a,b) \to N_2$ of $\pV^2$, there exists a decomposition $\sigma(t) = \sigma_+(t) * \sigma_-(t)$, where $\sigma_+:(a,b) \to \wh{M}_1$ and $\sigma_-:(a,b) \to \wh{M}_2$ are integral curves of the singular systems $\pV^2_\pm$.  Moreover, because the Vessiot group of this system is abelian, this decomposition may be constructed by quadrature, and it is unique up to the choice of constants of integration.  Then the corresponding 2-dimensional integral manifold $\iota_2: (a,b) \times (a,b) \to N_2$ is given by
\[ \iota_2(t_1, t_2) = \sigma_+(t_1) * \sigma_-(t_2). \]

Since the geometric Cauchy problem prescribes initial data for the coordinates $(x_1$, $x_2$, $x_3)$ (rather than $(y_1, y_2, y_3))$, the first step is to construct a non-characteristic integral curve $\sigma_0:(a,b) \to N_1$ of $\pV^1$ corresponding to this initial data.  Then the curve $\sigma = \psi \circ \sigma_0:(a,b) \to N_2$ will be the desired integral curve of $\pV^2$, from which we will construct the superposition formula.

From \eqref{char-systems-0}, we see that the Pfaffian system $\pV^1 = \pV^1_+ \cap \pV^1_-$ may be written as
\begin{equation}\label{Pfaffian-system-N1}
\begin{aligned}
\pV^1 = \Big{\{} & dv + \tfrac{1}{2} p_0^2\, dp - \tfrac{1}{2} q_0^2\, dq, \\
& dx_1 + \tfrac{1}{8} p_0^2(p_0^2(p^2+1) + q_0^2(q^2 - 2pq -1))\, dp - \tfrac{1}{4}p_0 q_0^2 (p-q)(pq+1)\, dp_0 \\
& \qquad - \tfrac{1}{8} q_0^2(p_0^2(p^2 - 2pq -1) + q_0^2(q^2+1))\, dq - \tfrac{1}{4}p_0^2 q_0 (p-q)(pq+1)\, dq_0, \\
& dx_2 + \tfrac{1}{8} p_0^2(p_0^2(p^2-1) + q_0^2(q^2 - 2pq +1))\, dp - \tfrac{1}{4}p_0 q_0^2 (p-q)(pq-1)\, dp_0 \\
& \qquad - \tfrac{1}{8} q_0^2(p_0^2(p^2 - 2pq +1) + q_0^2(q^2-1))\, dq - \tfrac{1}{4}p_0^2 q_0 (p-q)(pq-1)\, dq_0, \\
& dx_3 - \tfrac{1}{4}p p_0^2(p_0^2-q_0^2)\, dp + \tfrac{1}{4} p_0 q_0^2(p^2-q^2)\, dp_0 \\
& \qquad - \tfrac{1}{4}q q_0^2(p_0^2-q_0^2)\, dq + \tfrac{1}{4} p_0^2 q_0(p^2-q^2)\, dq_0 \Big{\}}.
\end{aligned}
\end{equation}
The general integral curve of this system may be constructed by choosing arbitrary functions $p, p_0, q, q_0: (a,b) \to \R$, substituting these functions into the Pfaffian system \eqref{Pfaffian-system-N1}, and then integrating the resulting expressions to obtain the remaining functions $v, x_1, x_2, x_3: (a,b) \to \R$; indeed, this was essentially the approach that we used to construct the general solution in \S \ref{construct-int-mflds-sec}.

For the geometric Cauchy problem, we must approach the construction of integral curves to the system \eqref{Pfaffian-system-N1} from a slightly different perspective.  Now we are given an initial curve $\g:(a,b) \to \R^{1,2}$, with parametrization
\begin{equation}\label{gamma-param}
\gamma(t) = (\bar{x}_1(t), \bar{x}_2(t), \bar{x}_3(t)),
\end{equation}
and a spacelike unit normal vector field $\bar{\ve}_3:(a,b) \to \R^{1,2}$ along $\gamma$. In order to lift the initial data $(\g, \bar{\ve}_3)$ to a non-characteristic integral curve $\sigma_0:(a,b) \to N_1$ of $\pV^1$, we must show how to obtain functions $p, p_0, q, q_0: (a,b) \to \R$ for which the functions $x_1, x_2, x_3: (a,b) \to \R$ and $\ve_3:(a,b) \to \R^{1,2}$ on the corresponding integral curve of $\pV^1$ agree with the given functions $\bar{x}_1, \bar{x}_2, \bar{x}_3$ and the vector field $\bar{\ve}_3$.

First, observe from \eqref{SO12-in-pq-coords} that
\begin{equation}\label{e3-in-pq-coords}
 \ve_3 = \left(\displaystyle{-\frac{p_0^2(p^2+1) - q_0^2(q^2+1)}{2 p_0 q_0 (p-q)} }, \ \displaystyle{-\frac{p_0^2(p^2-1) - q_0^2(q^2-1)}{2 p_0 q_0 (p-q)} }, \ \displaystyle{\frac{p_0^2 p - q_0^2 q}{p_0 q_0 (p-q)}} \right)^t.
\end{equation}
Replacing $\ve_3$ by the prescribed vector field $\bar{\ve}_3(t)$ leads to two {\em algebraic} constraints that must be satisfied by the four functions $p, p_0, q, q_0: (a,b) \to \R$.  Geometrically, these constraints may be interpreted as follows: One comes from the orthogonality requirement
\begin{equation}\label{orthog-requirement}
 \gamma'(t) \cdot \bar{\ve}_3(t) =  0,
\end{equation}
and once this is taken into account, specifying $\bar{\ve}_3$ is equivalent to specifying the ratio
\begin{equation}\label{specify-lambda}
 \lambda(t) = \frac{p_0(t)^2(p(t)^2+1) - q_0(t)^2(q(t)^2+1)}{p_0(t)^2(p(t)^2-1) - q_0(t)^2(q(t)^2-1)}
\end{equation}
between the first two components of $\bar{\ve}_3(t)$.  (Without loss of generality---for example, by applying an appropriate isometry of $\R^{1,2}$ and shrinking the interval $(a,b)$ if necessary---we may assume that $\bar{x}'_3(t) \neq 0$ for all $t \in (a,b)$, so that $\bar{\ve}_3(t) \neq (0,\ 0,\ 1)^t$.)
Rearranging, we see that \eqref{specify-lambda} is equivalent to the relation
\begin{equation}\label{4-fcts-relation}
p_0(t)^2 \left( (\lambda(t)-1) p(t)^2  - (\lambda(t) + 1)\right) = q_0(t)^2 \left( (\lambda(t)-1) q(t)^2  - (\lambda(t) + 1)\right),
\end{equation}
and, taking this relation into account, the orthogonality condition \eqref{orthog-requirement} becomes
\begin{equation}\label{2-fcts-relation}
(\lambda(t)-1) \,\bar{x}'_3(t)\, p(t)\, q(t) + (\lambda(t)\, \bar{x}'_1(t) - \bar{x}'_2(t))(p(t) + q(t)) + (\lambda(t)+1)\,\bar{x}'_3(t) = 0.
\end{equation}

The construction proceeds as follows: The algebraic solutions $(p, p_0, q, q_0):(a,b) \to \R^4$ of the relations \eqref{4-fcts-relation} and \eqref{2-fcts-relation} may be parametrized in terms of two arbitrary functions $r, s:(a,b) \to \R$ in a fairly straightforward way.  These expressions (along with the conditions $x_i = \bar{x}_i(t)$) may then be substituted into the last three 1-forms in \eqref{Pfaffian-system-N1} to obtain a system of three first-order ODEs for the two unknown functions $r, s$.  This system is redundant, but consistent, and hence may be written as a system of two first-order ODEs for $r$ and $s$, which in turn may be written as a single second-order ODE for one of the two functions, say $r$.  Any solution $r(t)$ of this ODE leads to functions $p, p_0, q, q_0:(a,b) \to \R^4$ which satisfy the last three 1-forms in \eqref{Pfaffian-system-N1}, and then the first 1-form in \eqref{Pfaffian-system-N1} may be used to construct the $v$-coordinate function by quadrature.  Once this has been accomplished, the corresponding isometric embedding may then be constructed as in \S \ref{construct-int-mflds-sec}.
Moreover, in a neighborhood of any point where $r'(t) \neq 0$, we may reparametrize the initial data with respect to $r$, after which the remainder of the process may be carried out via quadrature.  If we set $r(t) = t$, then the ODE for $r$ may be interpreted as a single ODE that the initial data must satisfy in order to be ``appropriately parametrized."  If the initial data satisfies this ODE, then the solution to the corresponding geometric Cauchy problem may be constructed entirely by quadrature.

Unfortunately, in general this construction is not practical to carry out explicitly.  Parametrizing the algebraic solutions to \eqref{4-fcts-relation} and \eqref{2-fcts-relation} is straightforward enough, but substituting the resulting expressions into \eqref{Pfaffian-system-N1} leads to a system that is computationally impractical to write down explicitly, even with the help of a computer algebra system such as {\sc Maple}.  However, for certain special choices of initial data, the algebra becomes tractable and we can construct explicit solutions.  The simplest case comes from choosing $\lambda(t) = 1$, in which case the relations \eqref{4-fcts-relation} and \eqref{2-fcts-relation} simplify considerably, to  $p_0(t) = q_0(t)$ and
\begin{equation}\label{2-fcts-relation-special}
 (\bar{x}'_1(t) - \bar{x}'_2(t))(p(t) + q(t)) + 2\,\bar{x}'_3(t) = 0,
\end{equation}
respectively.
Geometrically, this choice corresponds to requiring the normal vector $\bar{\ve}_3(t)$ to be contained in the lightlike plane in $T_{\gamma(t)}\R^{1,2}$ defined by $z_1 = z_2$, where $(x_i, z_i)$ are the canonical local coordinates on $T\R^{1,2}$.  This can be seen directly from the fact that when $p_0 = q_0$, the expression \eqref{e3-in-pq-coords} reduces to
\[ \ve_3 = \left(-\tfrac{1}{2}(p+q), \ -\tfrac{1}{2}(p+q), \ 1\right)^t. \]
Since we must have $\gamma'(t) \cdot \bar{\ve}_3(t) = 0$ and we have already imposed the requirement that $\bar{x}'_3(t) \neq 0$, this choice requires that our initial curve $\gamma$ satisfy the additional condition $\bar{x}'_1(t) \neq \bar{x}'_2(t)$. By applying an appropriate isometry of $\R^{1,2}$ and (if necessary) shrinking the interval $(a,b)$, we can ensure that these conditions hold for any initial curve $\gamma:(a,b) \to \R^{1,2}$.
Hence we can parametrize the algebraic solution space of \eqref{2-fcts-relation-special} as
\begin{equation}\label{pq-param}
 p(t) = \frac{\bar{x}'_3(t)}{(\bar{x}'_2(t) - \bar{x}'_1(t))} + s(t), \qquad q(t) = \frac{\bar{x}'_3(t)}{(\bar{x}'_2(t) - \bar{x}'_1(t))} - s(t)
\end{equation}
where $s:(a,b) \to \R$ is an arbitrary function.

Substituting $q_0(t) = p_0(t) = r(t)$ and the expressions \eqref{pq-param} into the system \eqref{Pfaffian-system-N1} yields the ODE system
\begin{equation}\label{ODE-sys-pp0qq0}
\begin{gathered}
r'(t)  = \frac{\bar{x}'_1(t) - \bar{x}'_2(t)}{2 r(t)^3 s(t)}, \\[0.1in]
s'(t)  = \frac{1}{2(\bar{x}'_1(t) - \bar{x}'_2(t)) r(t)^4}
\left(\frac{\bar{x}'_2(t)^2 + \bar{x}'_3(t)^2 - \bar{x}'_1(t)^2}{s(t)^2} - (\bar{x}'_1(t) + \bar{x}'_2(t))^2 \right), \\[0.1in]
v'(t) = -\frac{1}{2(\bar{x}'_1(t) - \bar{x}'_2(t)) r(t)^2}
\left(\frac{\bar{x}'_2(t)^2 + \bar{x}'_3(t)^2 - \bar{x}'_1(t)^2}{s(t)^2} - (\bar{x}'_1(t) + \bar{x}'_2(t))^2 \right).
\end{gathered}
\end{equation}
From the first equation in \eqref{ODE-sys-pp0qq0}, we can set
\begin{equation} \label{define-s}
s(t) = \frac{\bar{x}'_1(t) - \bar{x}'_2(t)}{2 r(t)^3 r'(t)},
\end{equation}
and then the second equation in \eqref{ODE-sys-pp0qq0} becomes a second-order scalar ODE for the function $r(t)$:
\begin{multline}\label{ode-for-p0}
({\bar{x}}'_1 - {\bar{x}}'_2)^4 r r'' - 4(({\bar{x}}'_1)^2 - ({\bar{x}}'_2)^2 - ({\bar{x}}'_3)^2) r^6 (r')^4
\\ + 2 ({\bar{x}}'_1 - {\bar{x}}'_2)^4 (r')^2 - ({\bar{x}}'_1 - {\bar{x}}'_2)^3 (\bar{x}''_1 - \bar{x}''_2) r r' = 0.
\end{multline}
Since the initial curve $\g$ satisfies the condition ${\bar{x}}'_1(t) - {\bar{x}}'_2(t) \neq 0$, the existence and uniqueness theorem for ODEs guarantees a local solution $r(t)$ to equation \eqref{ode-for-p0}.
Then taking $q_0(t) = p_0(t) = r(t)$, defining $s(t), p(t),q(t)$ as in \eqref{define-s} and \eqref{pq-param}, and then integrating the third equation in \eqref{ODE-sys-pp0qq0} to obtain $v(t)$ yields the desired integral curve $\sigma_0:(a,b) \to N_1$ for $\pV^1$.
Composing with $\psi$ gives the desired integral curve
$\sigma = \psi \circ \sigma_0:(a,b) \to N_2$  for $\pV^2$.

Finally, the decomposition $\sigma(t) = \sigma_+(t) * \sigma_-(t)$ is constructed as follows: Substitute the functions $p(t), p_0(t)$ defined by $\sigma$ into the singular system $\pV_+^2$ on $\wh{M}_1$.  The desired integral curve $\sigma_+:(a,b) \to \wh{M}_1$ of this system is obtained by integrating the resulting ODE system
\begin{equation}\label{V2plus-ODEs-again}
\begin{alignedat}{2}
(y^+_1)'(t) & = -\tfrac{1}{8}(p(t)^2+1) p_0(t)^4\, p'(t), & \qquad
(y^+_2)'(t) & = -\tfrac{1}{8}(p(t)^2-1) p_0(t)^4\, p'(t), \\
(y^+_3)'(t) & = \tfrac{1}{4}p(t) p_0(t)^4\, p'(t), & \qquad
(v^+)'(t) & = -\tfrac{1}{2} p_0(t)^2\, p'(t).
\end{alignedat}
\end{equation}
Similarly, the desired integral curve $\sigma_-:(a,b) \to \wh{M}_2$ of $\pV^2_-$ is obtained by integrating the ODE system
\begin{equation}\label{V2minus-ODEs-again}
\begin{alignedat}{2}
(y^-_1)'(t) & = \tfrac{1}{8}(q(t)^2+1) q_0(t)^4 \, q'(t), & \qquad
(y^-_2)'(t) & = \tfrac{1}{8}(q(t)^2-1) q_0(t)^4 \, q'(t), \\
(y^-_3)'(t) & = -\tfrac{1}{4}q(t) q_0(t)^4 \, q'(t), & \qquad
(v^-)'(t) & = \tfrac{1}{2} q_0(t)^2\, q'(t).
\end{alignedat}
\end{equation}
Initial conditions for both curves at some point $t_0 \in (a,b)$ should be chosen so that
\begin{equation}\label{V2initial_conditions}
y_i^+(t_0) + y_i^-(t_0) = y_i(t_0), \qquad v^+(t_0) + v^-(t_0) = v(t_0),
\end{equation}
where $y_i(t_0),\, v(t_0)$ are the values specified by $\sigma(t_0)$.

As mentioned earlier, if the solution to the ODE \eqref{ode-for-p0} satisfies $r'(t) \neq 0$, then locally we may reparametrize the initial curve $\gamma$ with respect to the variable $r$.  (Of course, this is rarely possible in practice, as it requires both solving the ODE explicitly and finding the inverse function of the solution.)  The components of the resulting curve (with $r(t) = t$) must then satisfy the relation
\begin{equation}\label{constraint-eqn}
-4\left( (\bar{x}'_1)^2 -(\bar{x}'_2)^2 - (\bar{x}'_3)^2 \right) t^6 + 2 (\bar{x}'_1 - \bar{x}'_2)^4  - (\bar{x}'_1 - \bar{x}'_2)^3 (\bar{x}''_1 - \bar{x}''_2) t = 0.
\end{equation}
Conversely, if the components of the given initial curve $\gamma$ satisfy \eqref{constraint-eqn}, then we may choose $r(t) = t$ and proceed as above, with the entire process requiring only quadrature to construct the solution.

In summary, we have proved the following theorem.

\begin{thm}\label{CauchyVessiotNF}
Let $\g = (\bar{x}_1, \bar{x}_2, \bar{x}_3): (a,b) \to \R^{1,2}$ be an immersed curve with $\bar{x}'_1(t) - \bar{x}'_2(t) \neq 0$ and $\bar{x}'_3(t) \neq 0$ for all $t \in (a,b)$, and let $t_0 \in (a,b)$.  Then:
\begin{enumerate}
\item There exists an interval $(\bar{a}, \bar{b}) \subset (a,b)$ containing $t_0$ and functions $p,p_0,q,q_0, v:(\bar{a}, \bar{b}) \to \R$, with $p_0 = q_0$, such that the curve $\s_0:(\bar{a}, \bar{b}) \to N_1$ defined by
\[ \sigma_0(t) = (p(t), q(t), p_0(t), q_0(t), \bar{x}_1(t), \bar{x}_2(t),  \bar{x}_3(t),  v(t)) \]
is a non-characteristic integral curve of the Pfaffian system $\pV^1$ on $N_1 = \phi(M \times\FRm)$, and such that for each $t \in(\bar{a}, \bar{b})$, the vector $\ve_3(t)$ determined by \eqref{e3-in-pq-coords} is contained in the lightlike plane in $T_{\gamma(t)}\R^{1,2}$ defined by $z_1 = z_2$, where $(x_i, z_i)$ are the canonical local coordinates on $T\R^{1,2}$.
\item These functions then determine a decomposition (unique up to constants of integration) $\sigma(t) = \sigma_+(t) * \sigma_-(t)$ of the curve $\sigma = \psi \circ \sigma_0:(\bar{a},\bar{b}) \to N_2$, where $\sigma_+:(\bar{a}, \bar{b}) \to \wh{M}_1$ and $\sigma_-:(\bar{a}, \bar{b}) \to \wh{M}_2$ are integral curves of the singular systems $\pV^1_\pm$, respectively, and hence a unique 2-dimensional integral manifold $\iota_2:(\bar{a}, \bar{b}) \times (\bar{a}, \bar{b}) \to N_2$ of $\pV^2$ given by $\iota_2(t_1, t_2) = \sigma_+(t_1) * \sigma_-(t_2)$.
\item The composition $\pi \circ \psi^{-1} \circ \iota_2: (\bar{a}, \bar{b}) \times (\bar{a}, \bar{b}) \to \R^{1,2}$ defines a local isometric embedding of an open subset of $(M, \bsy{g}_0)$ into $\R^{1,2}$ whose image contains the curve $\gamma((\bar{a}, \bar{b}))$.
\item If the component functions of $\gamma$ satisfy the relation \eqref{constraint-eqn}, then all these constructions may be performed using only quadratures.
\end{enumerate}

\end{thm}

A similar theorem could, in principle, be stated for initial data with an arbitrary normal vector field $\bar{\ve}_3:(a,b) \to \R^{1,2}$, although the appropriate analogs of the nondegeneracy condition $\bar{x}'_1(t) - \bar{x}'_2(t) \neq 0$ and the relation \eqref{constraint-eqn} are almost certainly impractical to determine explicitly.

We will conclude by illustrating the construction of the isometric embedding promised by Theorem \ref{CauchyVessiotNF} for a simple initial curve.

\begin{xmpl}\label{Example2}
Suppose we start with the curve
\[ \displaystyle \g(t)= (\bar{x}_1(t), \ \bar{x}_2(t),\ \bar{x}_3(t))=
\left(\frac{3t + 4t^3}{8}, \
\frac{3t - 4t^3}{8},\ \frac{3t^2}{4} \right) \]
for $t$ in some interval containing the initial point $t_0 = 1$.
We have
\begin{gather*}
 \bar{x}'_1(t) = \frac{3 + 12\,t^2}{8}, \qquad \bar{x}'_2(t) =  \frac{3 - 12\,t^2}{8}, \qquad \bar{x}'_3(t) = \frac{3t}{2}, \\
 \bar{x}''_1(t) = 3t, \qquad \bar{x}''_2(t) = -3t, \qquad \bar{x}''_3(t) = \frac{3}{2}.
\end{gather*}
This curve satisfies the constraint \eqref{constraint-eqn} as well as the conditions $\bar{x}'_1(t)-\bar{x}'_2(t)$, $\bar{x}'_3(t)\neq 0$, and so (taking $p_0(t) = q_0(t) = t$) from \eqref{pq-param}, \eqref{ODE-sys-pp0qq0}, and \eqref{define-s} we obtain
\[ p(t) = \frac{1}{t}, \qquad q(t) = -\frac{2}{t}, \qquad v(t) = \frac{3t}{2}. \]
Composing with $\psi$, we obtain the integral curve $\sigma:(a,b) \to N_2$ given by
\begin{align*}
\s(t) = \Big{(} & p(t) = \frac{1}{t}, \ p_0(t) = t, \ q(t) = -\frac{2}{t}, q_0(t) = t, \ v(t) = \frac{3t}{2}, \\
& y_1(t) = \frac{9t + t^3}{8}, \ y_2(t) = \frac{9t - t^3}{8}, \ y_3(t) = \frac{3t^2}{8} \Big{)}.
\end{align*}
Substituting these expressions for $p(t), p_0(t)$ into the ODE system \eqref{V2plus-ODEs-again} and integrating from $t=1$ to $t=t_1$, using the initial conditions
$$
y_i^+(1) = \tfrac{1}{2}y_i(1), \qquad v^+(1) = \tfrac{1}{2}v(1),
$$
yields
\begin{alignat*}{2}
y_1^+(t_1) &= \frac{3t_1 + t_1^3 + 11}{24} ,& \qquad
y_2^+(t_1) &= \frac{3t_1 - t_1^3 + 10}{24}, \\
y^+_3(t_1) &= \frac{5 - 2t_1^2}{16} ,& \qquad
v^+(t_1) & = \frac{2t_1+1}{4}.
\end{alignat*}
Similarly, substituting these expressions for $q(t), q_0(t)$ into the ODE system \eqref{V2minus-ODEs-again} and integrating from $t=1$ to $t=t_2$, using the initial conditions
$$y_i^-(1) = \tfrac{1}{2}y_i(1),\qquad v^-(1) = \tfrac{1}{2}v(1),$$ yields
\begin{alignat*}{2}
y_1^-(t_2) &= \frac{24 t_2 + 2 t_2^3 -11}{24} , &\qquad
y_2^-(t_2) &= \frac{12 t_2 - t_2^3 - 5}{12}, \\
y^-_3(t_2) &= \frac{8 t_2^2 - 5}{16}  ,&
\qquad
v^-(t_2) & = \frac{4t_2 - 1}{4} .
\end{alignat*}

Thus the superposition formula \eqref{sp-formula} yields
\begin{align*}
\iota_2(t_1, t_2) = \Big{(} & p = \frac{1}{t_1}, \ p_0 = t_1, \ q = -\frac{2}{t_2},\ q_0 = t_2, \ v = \frac{t_1 + 2 t_2}{2} , \\
& y_1 = \frac{3 t_1 + t_1^3 + 24 t_2 + 2 t_2^3}{24}, \
 y_2 = \frac{3 t_1 - t_1^3 + 24 t_2 - 2 t_2^3}{24}  , \
 y_3 = \frac{4 t_2^2 - t_1^2}{8}
\Big{)}.
\end{align*}
Finally, composing with $\psi^{-1}$ gives the following map $\iota_1 = \psi^{-1} \circ \iota_2:(a,b) \times (a,b) \to N_1$:
\begin{align*}
\iota_1(t_1, t_2) = \Big{(} & p = \frac{1}{t_1}, \ p_0 = t_1, \ q = -\frac{2}{t_2},\ q_0 = t_2, \ v = \frac{t_1 + 2 t_2}{2} , \\
& x_1 = \frac{18 t_2 - 9 t_1 + t_1^3 + 6 t_1^2 t_2 + 3 t_1 t_2^2 + 2 t_2^3}{24}, \\
& x_2 = \frac{18 t_2 - 9 t_1 - t_1^3 - 6 t_1^2 t_2 - 3 t_1 t_2^2 - 2 t_2^3}{24}, \\
& x_3 = \frac{3 t_1^2 + 3 t_2^2}{8}
 \Big{)}.
\end{align*}
The functions $x_i(t_1, t_2)$ are the components of the isometric embedding. The relationship between the variables $(u,v)$ on $(M,\bsy{g}_0)$ and $(t_1, t_2)$ can be deduced from the transformation $\phi^{-1}$ defined by equations (\ref{characteristicFirstIntegrals}), to be
$$
u=\frac{1}{2}p_0(t_1)\,q_0(t_2)\,(q(t_2)-p(t_1)),
$$
together with function $v$ above.
In this case the relationship turns out to be
$$
u= -\frac{2 t_1 + t_2}{2},\qquad
v= \frac{t_1 + 2 t_2}{2},
$$
which can be locally inverted to obtain the following expression for the explicit immersion in terms of $(u,v)$:
\begin{align*}
x_1 & = \frac{108 u + 135 v + 16 u^3 + 60 u^2 v + 48 u v ^2 + 20 v^3}{108}, \\
x_2 & = \frac{108 u + 135 v - 16 u^3 - 60 u^2 v - 48 u v ^2 - 20 v^3}{108}, \\
x_3 & = \frac{5 u^2 + 8 u v + 5 v^2}{6}.
\end{align*}
Some graphs of this surface, along with the initial curve $\gamma$, are shown in Figure \ref{Example2-fig}.  Note that the metric itself degenerates along the coordinate curve $u=0$, and this curve is precisely where the surface fails to be an immersion.

\begin{figure}[ht]
\begin{center}
\includegraphics[width=2in]{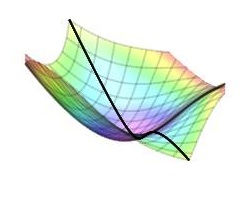}
\includegraphics[width=2in]{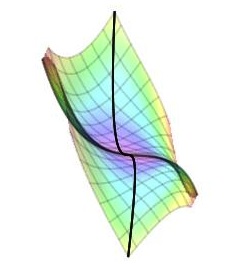}
\includegraphics[width=2in]{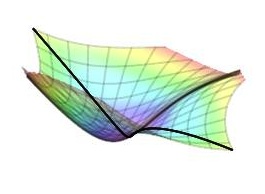}
\end{center}
\caption{\small Views of the surface of Example \ref{Example2}}
\label{Example2-fig}
\end{figure}

\end{xmpl}

\end{document}